\documentclass[10pt]{IEEEtran}

\usepackage{amsmath,amssymb,amsfonts}
\usepackage{times}
\usepackage{graphicx}
\usepackage{cite}
\usepackage{tabularx}
\usepackage{color}

\newlength{\figwidth}
\newlength{\figheight}
\newlength{\philwidth}

\def\mbb#1{\mathbb{#1}}%           blackboard font
\def\lrs#1{\ensuremath{\left[{#1}\right]}}%
\def\lrv#1{\ensuremath{\lvert{#1}\rvert}}%
\def\blrv#1{\ensuremath{\Bigg\lvert{#1}\Bigg\rvert}}%

\newcommand{\norm}[1]{\ensuremath{\|#1\|}}

\newcommand{\ZBZ}{{\mathbb Z}^B_0}
\newcommand{\babs}[1]{\blrv{#1}}
\newcommand{\abs}[1]{\lrv{#1}}%
\newcommand{\half}{\frac{1}{2}}
\def\Exp#1{\mbb{E}[{#1}]}%
\def\Prblr#1{\mbb{P}\kern-2pt\lrs{{#1}}}%

\newcommand{\perm}[2]{\left(\begin{array}{c}
                             #1 \\ #2
                             \end{array} \right)}
\newcommand{\prob}[1]{{\mathbb P}\{#1\}}

\newcommand{\floor}[1]{\lfloor #1 \rfloor}
\newcommand{\lc}{\left\{}
\newcommand{\rc}{\right\}}
\newcommand{\lb}{\left(}
\newcommand{\rb}{\right)}
\newcommand{\ls}{\left[}
\newcommand{\rs}{\right]}
\newcommand{\cmplx}{\mathbb C}

\newtheorem{cor}{Corollary}
\newtheorem{remark}{Remark}
\newtheorem{prop}{Proposition}

\newtheorem{lemma}{Lemma}
\newtheorem{theorem}{Theorem}
\newtheorem{definition}{Definition}
\newtheorem{example}{Example}
\newtheorem{conjecture}{Conjecture}

\definecolor{darkorange}{rgb}{1.00,0.50,0.00}
\begin{document}

\title{Eigenvalue Results for Large Scale Random Vandermonde Matrices with Unit Complex Entries}
\author{Gabriel H. Tucci and Philip A. Whiting
\thanks{
G.H. Tucci and P.A. Whiting
  are with Bell Laboratories,
  Alcatel--Lucent, 600 Mountain Ave, Murray Hill, NJ 07974.
  E-mail: gabriel.tucci@alcatel-lucent.com, pwhiting@research.bell-labs.com.}}

\maketitle

\begin{abstract}
This paper centers on the limit eigenvalue distribution
for random Vandermonde matrices with unit magnitude complex
entries. The phases of the entries are chosen independently and
identically distributed from the interval $[-\pi,\pi]$.  Various types
of distribution for the phase are considered and we establish the existence of the 
empirical eigenvalue distribution in the large matrix limit on a wide range of cases.
The rate of growth of the maximum eigenvalue is examined and shown to be no greater than
$O(\log N)$ and no slower than $O(\log N/\log\log N)$ where $N$ is the dimension of the matrix. Additional results include the
existence of the capacity of the Vandermonde channel (limit integral
for the expected log determinant).

\end{abstract}


\section{Introduction}%
\label{sec_intro}%
In this paper we consider random Vandermonde matrices with unit magnitude complex entries. Such matrices are defined as follows,
an $N\times L$ rectangular matrix ${\bf V}$ with unit complex entries
is a {\em Vandermonde matrix} if  there exist phases $\theta_1,\cdots,\theta_L \in [-\pi,\pi]$ such that
\begin{equation}
\label{eqn_Vandermondedefn}
{\bf V} := \frac{1}{\sqrt{N}}
\lb
\begin{array}{lcl}
1              & \ldots & 1 \\
e^{-j\theta_1} & \ldots & e^{-j\theta_L} \\
\vdots         & \vdots & \vdots \\
e^{-j(N-1)\theta_1} & \ldots & e^{-j(N-1)\theta_L} 
\end{array}
\rb.
\end{equation}
A random Vandermonde matrix is produced if the entries of the phase vector $\theta:= \lb \theta_1,\cdots, \theta_L \rb \in [-\pi,\pi]^L$
are  random variables. For the purposes of this paper it will be assumed that the phase vector has i.i.d. components, with distribution 
drawn according to a measure $\nu$ which has a density on $[-\pi,\pi]$. In other words, the measure $\nu$ is absolutely continuous with respect to the Lebesgue measure on the interval $[-\pi,\pi]$.

Central to this paper is the probability law of the eigenvalues
of the matrix ${\bf V V^*}$, or equivalently ${\bf V^{*}V}$, and related matrices. This work provides results for the empirical eigenvalue distribution of the random Vandermonde matrix, in particular the existence of the limit measure. The behaviour of the maximum eigenvalue is also considered and asymptotic upper and lower bounds are obtained in this case. 

Random Vandermonde matrices are a natural construction with a wide range of potential applications in such fields as finance \cite{Norberg}, signal processing \cite{SP}, wireless communications \cite{Porst}, statistical analysis \cite{Anderson}, security \cite{Sampaio} and medicine \cite{Strohmer}. This stems from the close relationship that unit magnitude complex Vandermonde matrices have with the discrete Fourier transform.  Amongst these is an important recent application for signal reconstruction using noisy samples (see \cite{SP}) where an asymptotic estimate is obtained for mean squared error. This asymptotic can be calculated as a random eigenvalue expectation, whose limit distribution depends on the signal dimension $d$.  In the case $d=1$ the limit is via random Vandermonde matrices. As $d\rightarrow \infty$ the Marchenko--Pastur limit distribution is shown to apply. Further applications were treated in \cite{GC02} including source identification and wavelength estimation.

The article \cite{GC02} is the principal reference to date for eigenvalue analysis of random Vandermonde matrices and this paper
builds on it. The results mentioned above represent some directions in which analysis of random Vandermonde matrices were progressed. In particular 
our results show the existence of the empirical eigenvalue measure for certain unbounded densities. We are aware of no published results for the maximum eigenvalue. 

The rest of the paper is as follows. In Section \ref{sec:rand_mat_ess} we give a brief overview of random matrix theory. This is followed
by Section \ref{sec:prelim} which sets up notation and presents some preliminary results which are used later in the paper.
Section \ref{Main} is used to prove the existence of the limit measure for a wide class of phase distributions with a density.
This is done via the well known method of moments \cite{Billingsley} which establishes the existence of the limit measure via Carleman's
theorem. The expansion coefficient limits are characterized both as integrals of products of sinc function as well as probabilities of
certain events involving independent uniform random variables. Combinatorial formulas are also given for the expansion coefficient
of certain classes of partitions. 

This is followed by the upper and lower bounds for the maximum eigenvalue of random Vandermonde matrices, see Section \ref{sec_maxeig}. In Section \ref{sec_lowerbnd} a conjecture for the lower bound on the Vandermonde expansion coefficients is stated. We also show some of the implications of this conjecture on the lower bound on the moment sequence. Applications are discussed in Section \ref{applications}. Finally numerical results are presented in Section \ref{sec_numerfin} and some conclusions are given. The remainder of the proof details are in the appendices.

\section{Random Matrix Essentials}
\label{sec:rand_mat_ess}
Throughout the paper we will denote by ${\bf A}^{*}$ the complex conjugate transpose of the matrix ${\bf A}$. ${\bf I}_{N}$ will represent the $N\times N$ identity matrix. We let $\mathrm{Tr}$ be the non--normalized trace for square matrices, defined by,
$$
\mathrm{Tr}({\bf A}):=\sum_{i=1}^{N}{a_{ii}},
$$
where $a_{ii}$ are the diagonal elements of the $N\times N$ matrix ${\bf A}$. We also let $\mathrm{tr}_{N}$ be the normalized trace, defined by $\mathrm{tr}_{N}({\bf A})=\frac{1}{N}\mathrm{Tr}({\bf A})$. Given two sequences of numbers $\{a_n\}_n$ and $\{b_{n}\}_n$ we say that $a_n\asymp b_{n}$ if $\lim_{n\to\infty}\frac{a_n}{b_n}=1$.

\par Let us consider a sequence $\{{\bf A}_{N}\}_{N\in\mathbb{N}}$ of selfadjoint $N\times N$ random matrices ${\bf A}_{N}$. In which sense can we talk about the limit of these matrices? It is evident that such a limit does not exist as an $\infty\times\infty$ matrix and there is no convergence in the usual topologies. What converges and survives in the limit are the moments of the random matrices. Let ${\bf A_{N}}=(a_{ij}(\omega))_{i,j=1}^{N}$ where the entries $a_{ij}$ are random variables on some probability space $\Omega$ equipped with a probability measure $P$. Therefore,
\begin{equation}
 \mathbb{E}(\mathrm{tr_{N}}({\bf A}_{N})):=\frac{1}{N}\sum_{i=1}^{N} \int_{\Omega}{a_{ii}(\omega)\,dP(\omega)}
\end{equation}
and we can talk about the $k$--th moment $\mathbb{E}(\mathrm{tr_{N}}({\bf A}_{N}^{k}))$ of our random matrix ${\bf A}_{N}$, and it is well known that for nice random matrix ensembles these moments converge as $N\to\infty$. So let us denote by $\alpha_{k}$ the limit of the $k$--th moment,
\begin{equation}
\alpha_{k}:=\lim_{N\to\infty}{\mathbb{E}(\mathrm{tr_{N}}({\bf A}_{N}^{k}))}.
\end{equation}
Thus we can say that the limit consists exactly of the collection of all these numbers $\alpha_{k}$. However, instead of talking about a collection of numbers we prefer to identify these numbers as moments of some random variable ${\bf A}$. Now we can say that our random matrices ${\bf A}_{N}$ converge to a random variable ${\bf A}$ in distribution (which just means that the moments of ${\bf A}_{N}$ converge to the moments of ${\bf A}$). We will denote this by ${\bf A}_{N}\to {\bf A}$.

\par One should note that for a selfadjoint $N\times N$ matrix ${\bf A} = {\bf A}^{*}$, the collection of moments corresponds also to a probability measure $\mu_{A}$ on the real line, determined by 
\begin{equation}
 \mathrm{tr}_{N}({\bf A}^{k})=\int_{\mathbb{R}}{t^{k}\,d\mu_{A}(t)}.
\end{equation}
This measure is given by the eigenvalue distribution of ${\bf A}$, i.e. it puts mass $\frac{1}{N}$ on each of the eigenvalues of ${\bf A}$ (counted with multiplicity):
\begin{equation}
 \mu_{{\bf A}}=\frac{1}{N}\sum_{i=1}^{N}{\delta_{\lambda_{i}}},
\end{equation}
where $\lambda_{1},\ldots,\lambda_{N}$ are the eigenvalues of ${\bf A}$. In the same way, for a random matrix ${\bf A}$, $\mu_{{\bf A}}$ is given by the averaged eigenvalue distribution of ${\bf A}$. Thus, moments of random matrices with respect to the averaged trace contain exactly that type of information in which one is usually interested when dealing with random matrices.

\begin{example} Let us consider the basic example of random matrix theory. Let ${\bf G}_{N}$ be the usual selfadjoint
Gaussian $N\times N$ random matrix (i.e., entries above the diagonal are independently and normally distributed). Then the famous theorem of
Wigner can be stated in our language in the form that
\begin{equation}\label{semicircular}
 {\bf G}_{N}\to s, \hspace{0.4cm}\text{where $s$ is a semicircular random variable},
\end{equation}
where semicircular just means that the measure $\mu_{s}$ is given by the semicircular distribution (or, equivalently, the even moments of the
variable $s$ are given by the Catalan numbers).
\end{example}

\par The empirical cumulative eigenvalue distribution function of an $N\times N$ selfadjoint random matrix ${\bf A}$ is defined by the random function 
$$
F_{{\bf A}}^{N}(\omega, x):=\frac{\#\{k\,:\,\lambda_{k}\leq x\}}{N}
$$
where $\lambda_{k}$ are the (random) eigenvalues of ${\bf A}(\omega)$ for each realization $\omega$. For each $\omega$ this function determines a probability measure $\mu_{N}(\omega)$ supported on the real line. These measures $\{\mu_{N}(\omega)\}_{\omega}$ define a Borel measure $\mu_{N}$ in the following way. Let $B\subset\mathbb{R}$ be a Borel subset then
$$
\mu_{N}(B):=\mathbb{E}\Big(\mu_{N}(\omega)(B)\Big).
$$

\par A new and crucial concept, however, appears if we go over from the case of one variable to the case of more variables. 
\vspace{0.3cm}
\begin{definition}
Consider $N\times N$ random matrices ${\bf A}_{N}^{(1)},\ldots,{\bf A}_{N}^{(m)}$ and variables $A_{1},\ldots,A_{m}$ (living in some abstract algebra $\mathcal{A}$ equipped with a state $\varphi$). We say that 
\begin{equation*}
 ({\bf A}_{N}^{(1)},\ldots,{\bf A}_{N}^{(m)})\to (A_{1},\ldots,A_{m})
\end{equation*}
in distribution if and only if 
\begin{equation}
 \lim_{N\to\infty}{\mathbb{E}\Big(\mathrm{tr_{N}}\Big({\bf A}_{N}^{(i_1)}\cdots {\bf A}_{N}^{(i_k)}\Big)\Big)}=\varphi(A^{(i_1)}\cdots A^{(i_k)})
\end{equation}
for all choices of $k$, $1\leq i_{1},\ldots, i_{k}\leq m$.
\end{definition}

\vspace{0.3cm}
The $A_{1},\ldots,A_{m}$ arising in the limit of random matrices are a priori abstract elements in some algebra $\mathcal{A}$, but it is good to know
that in many cases they can also be concretely realized by some kind of creation and annihilation operators on a full Fock space \cite{Voi2}. Indeed,
free probability theory was introduced by Voiculescu for investigating the structure of special operator algebras generated by these type of
operators. In the beginning, free probability had nothing to do with random matrices.
\vspace{0.3cm}
\begin{example}
Let us now consider the example of two independent Gaussian random matrices ${\bf G}_{N}^{(1)},{\bf G}_{N}^{(2)}$ (i.e., each of them is a selfadjoint
Gaussian random matrix and all entries of ${\bf G}_N^{(1)}$ are independent from all entries of ${\bf G}_{N}^{(2)}$ ). Then one knows that all joint moments converge, and we can say that 
\begin{equation}
({\bf G}_{N}^{(1)},{\bf G}_{N}^{(2)})\to (s_1,s_2),
\end{equation} 
where Wigner's Theorem tells us that both $s_1$ and $s_2$ are semicircular. The question is: What is the relation between $s_1$ and $s_2$? Does the independence between ${\bf G}_N^{(1)}$ and ${\bf G}_{N}^{(2)}$ survive in some form also in the limit? The answer is yes and is
provided by a basic theorem of Voiculescu which says that $s_1$ and $s_2$ are free. For a formal definition of freeness and more results in free probability see \cite{Voi2}, \cite{Voi1}, \cite{Speicher}, \cite{Speicher2} and \cite{Verdu}.
\end{example}

\par Let $\{{\bf D}_{r}(N)\}_{r=1}^{n}$ be a set of non--random diagonal $L\times L$ matrices, where we implicitly assume that $\frac{L}{N}\to c$. 
As we previously discuss the family has a joint limit distribution as $N\to\infty$ if the limit
\begin{equation}
 D_{i_{1},\ldots,i_{s}}=\lim_{N\to\infty}{\mathrm{tr_{L}}\Big({\bf D}_{i_1}(N)\cdots {\bf D}_{i_s}(N) \Big)}
\end{equation}
exists for all choices of $i_{1},\ldots,i_{s}\in\{1,\ldots,n\}$.

\section{Vandermonde Expansion Coefficients}
\label{sec:prelim}
\label{Main}
\par We will denote by $\mathcal{P}(n)$ the set of all partitions of $\{1,\ldots,n\}$, and use $\rho$ as a notation for a partition in $\mathcal{P}(n)$. Also, we will write $\rho=\{B_{1},\ldots,B_{r}\}$, where $B_{j}$ will be used to denote the blocks of $\rho$ and $r$ is the number of blocks in $\rho$. Let $\mathcal{I}_{j}=\{i_{j1},i_{j2},\ldots,i_{j|B_j|}\}$ be the set of elements in the block $B_{j}$. We denote by $D_{\rho}$ 
\begin{equation}
 D_{\rho}:=\prod_{j=1}^{k}{D_{B_j}}
\end{equation}
where
\begin{equation}
 D_{B_{j}}=D_{i_{j1},i_{j2},\ldots,i_{j|B_j|}}.
\end{equation}

\par Consider an $N\times L$ random Vandermonde matrix with unit complex entries as given in equation (\ref{eqn_Vandermondedefn}). We will be considering the case where the phases $\theta_{1},\ldots,\theta_{L}$ are independent and identically distributed taking values in $[-\pi,\pi]$. The variables $\theta_{i}$ will be called the phase distributions and we will denote by $\nu$ their probability distribution. It was proved in \cite{GC02} that if $d\nu=f(x)\,dx$ for $f(x)$ continuous in $[-\pi,\pi]$ then the matrices ${\bf V^{*}V}$ have finite asymptotic moments. In other words, the limit 
\begin{equation}
m_{n}^{(c)}=\lim_{N\to\infty}{\mathbb{E}\Big[\mathrm{tr_{L}}\Big(({\bf V^{*}V})^n \Big)\Big]}
\end{equation}
exists for all $n\geq 0$. Moreover, $m_{n}^{(c)}$ is equal to 
\begin{equation}
m_{n}^{(c)}=\sum_{\rho\in\mathcal{P}(n)}{K_{\rho,\nu}c^{|\rho|-1}}
\end{equation}
where $K_{\rho,\nu}$ are positive numbers indexed by the partitions $\rho$. We call these numbers {\it Vandermonde expansion coefficients}. 

The fact that all the moments exists is not enough to guarantee that there exists a limit probability measure having these moments. However, we will prove that in this case this is true. In other words, the matrices ${\bf V^{*}V}$ converge in distribution to a probability measure $\mu_{c}$ supported on $[0,\infty)$ where $c=\lim{\frac{L}{N}}$ as the dimension grows. More precisely, let $\mu_{L}$ be the empirical eigenvalue distribution of the $L\times L$ random matrices ${\bf V^{*}V}$. Then $\mu_{L}$ converge weakly to a unique probability measure $\mu_{c}$ supported in $[0,+\infty)$ with moments 
$$
m_{n}^{(c)}=\int_{0}^{\infty}{t^{n}\,d\mu_{c}(t)}.
$$ 
We also enlarge the class of functions for where the limit eigenvalue distribution exists to include unbounded densities and we found lower bounds and upper bounds for the maximum eigenvalue.

Even though these numbers are indexed by the set of partition $\mathcal{P}(n)$ they do not have the properties of classical cumulants (these latter are the coefficients of the log characteristic function). Instead these numbers are weights which tell us how a partition in the moment formula should be computed. They are also to be distinguished from the free cumulants in free probability theory which are the coefficients of the $R$--transform (see \cite{Voi1} and \cite{Voi2} for more details on this).

\par  Let $\rho$ be a partition of the set $\{1,2,\ldots,n\}$ with $r$ blocks. Consider the map $\rho:\{1,2,\ldots,n\}\to\{1,2,\ldots,r\}$ such that $\rho(i)=k$ if and only if $i$ is in block $k$. The analysis done in \cite{GC02} and \cite{SP} with respect to these coefficients can be summarized as follows,
\begin{equation}
 K_{\rho,u}=\lim_{N}{K_{\rho,u,N}}
\end{equation}
where 
\begin{eqnarray}\label{eqq}
K_{\rho,u,N} & = & \frac{1}{N^{n+1}}\cdot\frac{N!}{(2\pi)^{r}(N-r)!}\cdot \\
& \cdot & \int_{[-\pi,\pi]^{r}}{F_{\rho}(x_{1},\ldots,x_{r})\,dx_{1}\ldots dx_{r}} \nonumber 
\end{eqnarray}
with 
$$
F_{\rho}(x_{1},\ldots,x_{r}):=\prod_{j=1}^{n}{F(x_{\rho(j)}-x_{\rho(j+1)})}
$$
and
$$
F(x)=\frac{\sin(\frac{N}{2}x)}{\sin(\frac{x}{2})}.
$$

\vspace{0.2cm}
\par Consider a random Vandermonde matrix defined as before with phase distribution $\nu$ concentrated in the interval $[-\pi,\pi]$.  For
each $q\geq 1$ let $L^{q}([-\pi,\pi])$ be the class of all the $q$--integrable functions with respect to Lebesgue measure. We will
denote by $\mathcal{L}:=\cap_{q\geq 1}{L^{r}([-\pi,\pi])}$ the space of all $q$--integrable functions. Note that
$L^{\infty}([-\pi,\pi])\subsetneq\mathcal{L}$. Further denote by ${\cal A}$ those probability measures supported on $[-\pi,\pi]$ with non--negative Fourier series coefficients. More specifically, 
$$
\phi(m):=\int_{-\pi}^{\pi}{e^{-imt}\,d\nu(t)}\geq 0
$$
for all $m\in\mathbb{Z}$.

\vspace{0.3cm}

\begin{theorem}\label{t1}
If $\nu$ is a probability measure which is absolutely continuous with respect to Lebesgue measure and with probability
density $p$ with $p\in\mathcal{L}\cap {\cal A}$, then for each partition $\rho$ the expansion coefficient $K_{\rho,\nu}$ exists.  Moreover, if $\rho$ is a partition of $n$ elements with $r$ blocks then
\begin{equation}
 K_{\rho,\nu}=K_{\rho,u}\cdot (2\pi)^{(r-1)}\cdot\int_{-\pi}^{\pi}{p(x)^r\,dx}
\end{equation}
where $K_{\rho,u}$ is the expansion coefficient with respect to the uniform distribution.
\end{theorem}

\vspace{0.3cm}
\noindent The proof of this Theorem can be found in the Appendix.

\vspace{0.3cm}
\begin{remark}
A close look at the proof of this Theorem shows it is enough for the proof to work that the following holds,
\begin{equation}
\label{eqn_dominate}
\sum_{{\bf m}\in\mathbb{Z}^{k}_{0}} \prod_{n=1}^k \vert a(m_n) \vert < \infty
\end{equation}
for every $k\geq 1$ where 
$$
\mathbb{Z}^{k}_{0}:=\{(m_1,\ldots,m_k)\,:\,m_1+\ldots+m_k=0\}
$$ 
and 
$$
a(m):=\mathbb{E}(e^{im\theta_{k}})=\int_{-\pi}^{\pi}{e^{imx}p(x)\,dx}.
$$
The case $k=1$ is trivial and $k=2$ is essentially Parseval's identity. However, for $k\geq 3$ the condition of $p$ being in $\mathcal{L}$ alone is necessary but not sufficient as was pointed out in \cite{Tao}. Consider for instance a random sign function at scale $\frac{1}{N}$ for some $N$, i.e. a function of the form 
$$
f(x)=\sum_{n=1}^{N-1}{\epsilon_n \mathbf{1}_{[n/N,(n+1)/N)}}
$$ 
for some i.i.d. signs $\epsilon_n$. This function is bounded in every $L^{q}$ but Khintchine's inequality tells us that the first $N$ Fourier coefficients of this function are of size $\frac{1}{\sqrt{N}}$ on the average, which will lead to the  divergence of the above series for $k\geq 3$ in the limit $N\to\infty$. One can formalize this divergence by creating a suitable linear combination of the above samples over all $N$, or appealing to the uniform boundedness principle.
\end{remark}

\vspace{0.2cm}
\par The following proposition give us an alternative way to compute the Vandermonde expansion coefficients.

\vspace{0.2cm}
\begin{prop}\label{p1} 
The Vandermonde mixed moment expansion coefficient $K_{\rho,u}$ can also be written as
 \begin{equation}
  K_{\rho,u}=\int_{\mathbb{R}^{(r-1)}}{G_{\rho}(x_{1},\ldots,x_{r})\,dx_{1}\ldots dx_{r-1}}
 \end{equation}

\noindent where 
\begin{equation}
 G_{\rho}(x_{1},\ldots,x_{r}):=\prod_{j=1}^{n}{G(x_{\rho(j)}-x_{\rho(j+1)})}
\end{equation}
and
$$G(x)=\frac{\sin(\pi x)}{\pi x}$$
\end{prop}

\vspace{0.2cm}
\noindent The proof of this proposition follows easily from equation (\ref{eqq}).

\vspace{0.2cm}
\begin{remark}
Note that in the previous Proposition we need to integrate only with respect to $r-1$ variables even though $G$ is a function of $r$ variables. The reason for this is that $G$ is a function of the differences.
\end{remark}

\vspace{0.3cm}
\noindent The following is an example of a family of probability distributions that do not belong to the family $\mathcal{L}$, see also \cite{GC02}.
\vspace{0.3cm}
\begin{example} Let $0<\alpha<1$ and let 
\begin{equation}
\label{eqn_badpdf}
p_{\alpha}(x):=\frac{1-\alpha}{2\pi^{1-\alpha}}\cdot\frac{1}{\abs{x}^{\alpha}}\cdot\mathbf{1}_{[-\pi,\pi]}(x).
\end{equation}
Then it is easy to see that 
$$
\int_{-\pi}^{\pi}{p(x)\,dx}=1
$$
and 
$$
\int_{-\pi}^{\pi}{p(x)^{n}\,dx}<+\infty \iff 0\leq n<\frac{1}{\alpha}.
$$
The density is depicted in Figure 
\ref{fig_palpha},
\begin{figure}[!Ht]
  \begin{center}
\hspace{-0.2in}
    \includegraphics[width=3.50in]{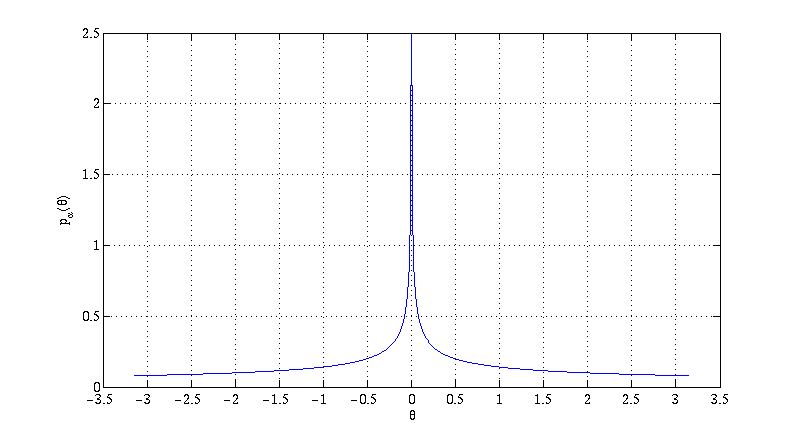}
    \caption{Density $p_{\frac{1}{2}}(\theta)$ as in Example 3. As we observed this density does not meet the convergence criterion.}
    \label{fig_palpha}
  \end{center}
\end{figure}
for $\alpha = \half$, so the limiting expansion coefficient does not exist for $n \geq 2$.
\end{example}

\vspace{0.3cm}
\noindent We now provide an example of an unbounded pdf living in the family $\mathcal{L}\cap\mathcal{A}$ for which Theorem \ref{t1} applies.

\vspace{0.3cm}

\begin{lemma}
\label{lemma_logpdf}
Suppose that $f\in L^{1}([-\pi,\pi])$ is a symmetric density around 0, convex and decreasing on $(0,\pi]$. Then the Fourier series coefficients are non--negative. 
\end{lemma}

\begin{proof}
The Fourier series coefficients exist since $f$ is an integrable function by hypothesis. By the symmetry hypothesis the Fourier coefficients are real and equal to $2\int_{0}^{\pi}{f(x)\cos(mx)\,dx}$ for $m\in\mathbb{Z}$. Note that since $f(x)$ is even it is enough to verify the claim for $m\geq 1$. Since convex functions are differentiable a.e., see \cite{Rockafeller}, integration by parts applies yielding
\begin{eqnarray*} 
\int_{0}^{\pi} {f(x)\cos(mx)\,dx}  & = &\ls f(x)\frac{\sin(mx)}{m} \rs_0^\pi \\ 
& - & \int_{0}^{\pi}{f'(x)\frac{\sin(mx)}{m}\,dx}.\\
\end{eqnarray*}
The first term is 0 under the assumption that $f\in L^{1}$ ($\lim_{x\to 0}{xf(x)}=0$) and the last
term is positive since $-f'(x) \geq 0$ and it is decreasing due to the convexity.
\end{proof}

\vspace{0.3cm}

\begin{example}
We now show that the unbounded pdf defined as
\begin{equation}
\label{eqn_goodpdf}
f(x):= \frac{1}{2\pi} \log \frac{\pi}{\vert x \vert}
\end{equation}
does satisfy
$$
\int_{-\pi}^\pi  (f(x))^n dx < \infty,
$$
for all $n\geq 0$. To see this first note that $f$  is convex on $(0,\pi]$. It follows that the infinite Riemann sums for
$x_k =\frac{1}{k}, k=1,2,3,\cdots$ give rise to the following double 
inequality for the integral of $f$,
$$
\sum_{k=1}^\infty \frac{f(k)^n}{k(k+1)}
< \int_0^\pi (f(x))^n \,dx <  \sum_{k=1}^\infty
\frac{f(k+1)^n}{k(k+1)} < \infty
$$
for all $n\geq 0$. Therefore $f\in\mathcal{L}$. Applying the previous Lemma we see immediately that $f$ also belongs to the family $\mathcal{A}$ and hence Theorem \ref{t1} applies to the probability distribution with density $f(x)$. Moreover we can compute its Fourier coefficients as 
$$
\widehat{f}(m):=\frac{1}{\pi}\int_{0}^{\pi}{\log\frac{\pi}{x}\cos(mx)\,dx}=\frac{\mathrm{Si}(|m|\pi)}{|m|\pi}
$$
where  
$$
\mathrm{Si}(t):=\int_{0}^{t}{\frac{\sin(x)}{x}\,dx}>0
$$
for all $t>0$. The density itself is depicted in Figure \ref{fig_logpdf}.
\begin{figure}[!Ht]
  \begin{center}
  \hspace{-0.2in}
    \includegraphics[width=3.5in]{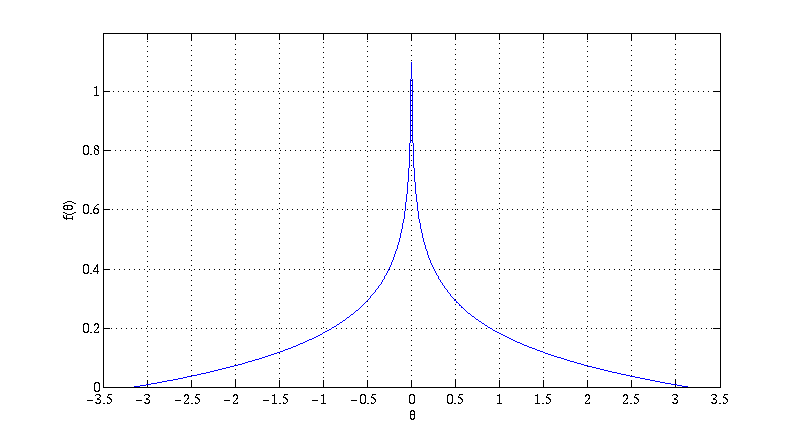}
    \caption{Unbounded density $f(\theta)$ as in Example 4 for which the limit eigenvalue distribution exists.}
    \label{fig_logpdf}
  \end{center}
\end{figure}
\end{example}

\vspace{0.3cm}
\noindent An important consequence of Theorem \ref{t1} is that it 
gives the uniform phase distribution a distinctive role. See also \cite{GC02} for an alternative proof of the following Proposition.

\vspace{0.3cm}
\begin{prop}
Let ${\bf V}_{\nu}$ denote a random Vandermonde matrix with phase distribution $\nu$, and set 
$$
m_{\nu,n}:=\lim_{N\to\infty}{\mathbb{E}\Big[\mathrm{tr}_{L}\Big( ({\bf V}^{*}_{\nu}{\bf V}_{\nu})^{n}\Big) \Big]}.
$$
Then we have that for all $n\geq 1$
$$m_{u,n}\leq m_{\nu,n}.$$
\end{prop}

\begin{proof}
It is enough to prove that for every partitions $\rho\in\mathcal{P}(n)$, with an arbitrary number of blocks $r$,  $K_{\rho,u}\leq K_{\rho,\nu}$. 
By the previous Theorem 
$$
K_{\rho, \nu}=K_{\rho,u}\cdot (2\pi)^{(r-1)}\cdot\int_{-\pi}^{\pi}{p(x)^r\,dx}.
$$
By Jensen's inequality 
$$
\frac{1}{2\pi}\int_{-\pi}^{\pi}{p(x)^{r}\,dx}=\mathbb{E}\big( p(x)^{r}\big)\geq \mathbb{E}\big(p(x)\big)^{r}=\Bigg(\frac{1}{2\pi}\Bigg)^{r}.
$$
Therefore, $(2\pi)^{r-1}\int_{-\pi}^{\pi}{p(x)^{r}\,dx}\geq 1$ and it follows that $K_{\rho,u}\leq K_{\rho,\nu}$.
\end{proof}

\vspace{0.3cm}
\noindent Using Theorem \ref{t1} and Theorem 1 from \cite{GC02} we can state the following result.

\vspace{0.3cm}
\begin{theorem}
Assume that $\{{\bf D}_{r}(N)\}_{1 \leq r\leq n}$ have a joint limit distribution as $N\to\infty$. Assume also that $\nu$ satisfies the hypothesis of Theorem \ref{t1}. Then the limit 
$$
M_{n}=\lim_{N}{\mathbb{E}\Big[\mathrm{tr}_{L}\Big( {\bf D}_{1}(N){\bf V^{*}V}D_{2}(N){\bf V^{*}V}\ldots {\bf D}_{n}(N){\bf V^{*}V} \Big)\Big]}
$$
exists when $\frac{L}{N}\to c$ and is equal to 
\begin{equation}
 \sum_{\rho\in\mathcal{P}(n)}{K_{\rho,\nu}c^{|\rho|-1}D_{\rho}}.
\end{equation}
\end{theorem}

\vspace{0.2cm}
In the following examples we will compute $K_{\rho,u}$ for some families of partitions.

\vspace{0.3cm}
\begin{example}
Let $n$ even and let $\rho$ be the partition 
$$
\rho=\{\{1,3,\ldots, n-1\},\{2,4,\ldots,n\} \}.
$$ 
Then by Proposition \ref{p1}
$$K_{\rho,u}=\int_{-\infty}^{+\infty}{G(x-y)^{\frac{n}{2}}G(y-x)^{\frac{n}{2}}\,dx}=\Big( G^{\frac{n}{2}}\ast G^{\frac{n}{2}} \Big)(0)$$

\noindent where $G(x)=\frac{\sin(\pi x)}{\pi x}$. Alternatively, since $G(x)=G(-x)$ we can write 
$$K_{\rho,u}=\int_{-\infty}^{+\infty}{\Bigg( \frac{\sin(\pi x)}{\pi x} \Bigg)^{n}\,dx}.$$

\noindent This integral can be easily computed giving us the following result:
\begin{equation}
K_{\rho,u}=\frac{1}{(n-1)!}\cdot\sum_{k=0}^{\frac{n}{2}-1}{(-1)^{k}{n\choose k}\Big(\frac{n}{2}-k\Big)^{n-1}}
\end{equation}

\noindent Evaluating the last expression we see that $K_{\{ \{1,3\}, \{2,4\}\},u}=\frac{2}{3}$ and $K_{\{ \{1,3,5\},\{2,4,6\}\},u}=\frac{11}{20}$.
\end{example}

\vspace{0.3cm}

\begin{example}
Let $n\geq 1$ and consider the partition 
$$
\rho=\{\{1,n+1\}, \{2,n+2\},\ldots \{n,2n\}\}
$$ 
then $\rho\in\mathcal{P}(2n)$ and $|\rho|=n$. Then
\begin{eqnarray*}
K_{\rho,u} & = & \int_{\mathbb{R}^{n-1}}{\prod_{i=1}^{n}{G(x_{i}-x_{i+1})^{2}}\,dx_{1}\ldots d{x_{n-1}}}\\
& = &\Big( G^{2}\ast G^{2}\ast\ldots\ast G^{2}\Big)(0).
\end{eqnarray*}

\noindent Using the fact that the Fourier transform of $G^{2}(x)$ is the triangular function $\mathrm{tri}(t):=\max(1-|t|,0)$ and elementary properties of the Fourier transform we see that 
$$
K_{\rho,u}=\int_{-\infty}^{+\infty}{(\mathrm{tri}(t))^{n}\,dt}=\frac{2}{n+1}
$$
for all $n\geq 1$.
\end{example}

\section{Bell Numbers, Harper's Theorem and Existence of the Limiting Measure}
\label{Bell}
The numbers $B_n$ are defined to be the Bell numbers, i.e. the number of possible ways in which the set $\lc 1, \cdots, n \rc$ (or any set of size $n$) can be 
partitioned into distinct subsets. Further define $s(k,n)$ to be the Stirling numbers of the second kind, i.e. the number of partitions
of a set of size $n$ into $k$ subsets. By definition 
$$
\sum_{k=1}^{n}{s(k,n)} = B_n.
$$

If we normalize the $s(k,n)$ by $B_n$ we obtain a probability distribution. The following result establishes the asymptotic normality of this
distribution.
\begin{theorem}{[Harper]}\cite{Harper}\label{Harper}
The Stirling numbers of the second kind are asymptotically normal in the sense that 
$$B_{n}^{-1}\sum_{j=1}^{x_{n}}{s(j,n)}\to \frac{1}{\sqrt{2\pi}}\cdot\int_{-\infty}^{x}{e^{-\frac{t^2}{2}}\,dt}\hspace{0.3cm}\text{as}\,\,\,\,n\to +\infty$$
\noindent where $x_{n}=\sigma_{n}x+(B_{n+1}/B_{n}-1)$ and $\sigma_{n}=\big(B_{n+2}/B_{n}-(B_{n+1}/B_{n})^{2}-1\big)^{\frac{1}{2}}$.

\vspace{0.3cm}
\noindent Moreover, if we define $J_{n}$ the integer such that $s(J_{n},n):=\max_{1\leq i\leq n}{s(i,n)}$ then $J_{n}\asymp \frac{n}{e\log{n}}$ and 
$$s(J_{n},n)\asymp \frac{1}{\sqrt{2\pi}}\cdot\frac{B_{n}}{\sigma_{n}}.$$
\end{theorem}

\vspace{0.3cm}

\begin{remark}\label{r3}
We would like also to point out that as a Corollary of Lemma 2 in \cite{Harper} 
$$
\lim_{n}{\sigma_{n}^{\frac{1}{n}}}=1.
$$ 
This result will be used later on this work.
\end{remark}

\vspace{0.3cm}

An estimate of de Brujin (1981), and also an estimate of Moser and Wyman \cite{55MoserWyman} states that,
\begin{equation*}
B_{n}\asymp \mathrm{exp}\Bigg(n\log n-n\log\log n -n + \frac{n}{\log n}\log\log n + \frac{n}{\log n}\Bigg).
\end{equation*}

\noindent It follows that
$$
B_n^{\,\frac{1}{n}}\asymp \mathrm{exp}\Bigg(\log n-\log\log n -1 + \frac{\log\log n}{\log n} + \frac{1}{\log n}\Bigg)\leq n.
$$

\par Let ${\bf V}$ be an $N\times L$ random Vandermonde matrix with phase probability distribution $\nu$. Assume also that $\nu$ is absolutely continuous with respect to Lebesgue measure and with continuous density $p(x)$. It was proved in Theorem 3 of \cite{GC02} that the Vandermonde expansion coefficient $K_{\rho,\nu}$ exists and that 
\begin{equation}
m_{\nu,n}^{(c)}:=\lim_{N\to\infty}{\mathbb{E}\Big[\mathrm{tr_{L}}\Big(({\bf V^{*}V})^n \Big)\Big]}=\sum_{\rho\in\mathcal{P}(n)}{K_{\rho,\nu}c^{|\rho|-1}}
\end{equation}
for all $n\geq 1$. However, this is not enough to guarantee the existence or uniqueness of probability measure supported in $[0,+\infty)$ having these moments.  Let $\rho\in\mathcal{P}(n)$ be a partition with $r$ blocks then
$$
K_{\rho,\nu}=K_{\rho,u}\cdot (2\pi)^{(r-1)}\cdot\int_{-\pi}^{\pi}{p(x)^r\,dx}\leq (2\pi\cdot\norm{p}_{\infty})^{n}
$$ 
since $0<K_{\rho,u}\leq 1$. Therefore, 
$$
m_{\nu,n}^{(c)}\leq (2\pi\cdot\norm{p}_{\infty}\cdot\max\{c,1\})^{n}\cdot B_{n}
$$
for all $n\geq 1$.
Let $C=(2\pi\cdot\norm{p}_{\infty}\cdot\max\{c,1\})^{\frac{1}{2}}$ and define 
$$
\beta_{n}:=\inf_{k\geq n}{\Big(m_{\nu,k}^{(c)}\Big)^{\frac{1}{2k}}}\leq C\cdot\inf_{k\geq n}{B_{k}^{\frac{1}{2k}}}\leq C\cdot\inf_{k\geq n}{\sqrt{k}}=C\cdot\sqrt{n}.$$

\noindent Hence, $\beta_{n}^{-1}\geq C^{-1}\cdot n^{-1/2}$ and therefore
$$
\sum_{n=1}^{+\infty}{\beta_{n}^{-1}}=+\infty.
$$
Therefore, by Carleman's Theorem \cite{Carleman} there exists a unique probability measure $\mu_{c}$ supported on $[0,+\infty)$ such that 
$$
m_{\nu,n}^{(c)}=\int_{0}^{+\infty}{t^{n}\,d\mu_{c}(t)},\hspace{0.4cm}\forall n\geq 1.
$$
In other words,  the sequence $m_{\nu,n}^{(c)}$ is {\em distribution determining}. Indeed, let $\mu_L$ be the empirical eigenvalue measure distribution of the $L\times L$ matrix ${\bf V^{*}V}$. We have thus proved the following result.

\vspace{0.3cm}

\begin{theorem}
\label{thm_limitdistn}
The sequence of distributions $\mu_L$ converge to a unique limiting distribution $\mu_{c}$ for which all positive moments exist and if $\Lambda$ is a positive random variable with distribution $\mu_{c}$ then for all $n\geq 1$,
$$
\Exp{\Lambda^n} = m_{\nu,n}^{(c)}.
$$
\end{theorem}

\section{Maximum Eigenvalue}
\label{sec_maxeig}
Let ${\bf V}$ be a square $N\times N$ random Vandermonde matrix with phase distribution $\nu$. In this Section we will focus on the study of growth of the maximum eigenvalue of the matrix ${\bf V^{*}V}$ as a function of $N$. More specifically, we know that the matrix ${\bf V^{*}V}$ is a positive definite $N\times N$ random matrix with eigenvalues $\{\lambda_{1}\leq\lambda_{2}\leq\ldots\leq\lambda_{N}\}$. It is clear that $\lambda_{N}$ is a random variable with values in the interval $[0,N]$. The value $N$ is taken by the random variable $\lambda_{N}$ in the event in which all the phases in the random matrix are equal ($\theta_{1}=\theta_{2}=\ldots=\theta_{N}$) and this event has zero probability. First we will prove an upper bound on the expectation $\mathbb{E}[\lambda_{N}]$.

\vspace{0.2cm}

\noindent Note that, of course, such a study does not rely on the existence of a limiting measure. 

\subsection{Upper bound}
\label{max_upper}

We first note that the matrix ${\bf V^{*}V}$ has the same eigenvalues as the matrix 
$$
{\bf X}_{N}=\Bigg( \frac{\sin(\frac{N}{2}(\theta_{i}-\theta_{j}))}{N\sin(\frac{\theta_{i}-\theta_{j}}{2})}\Bigg)_{1\leq i,j\leq N}.
$$
See Appendix A for a proof of this statement. It is a well known result in Linear Algebra (see \cite{Horn}) that 
$$
\lambda_{N}\leq (\max_{i}\{c_{i}\}\max_{i}\{r_{i}\})^{\frac{1}{2}}
$$ 
where 
$c_i$ and $r_i$ are the $l_{1}$ norms of the columns and rows of the matrix. In this particular case,
$$
r_{i}:=\sum_{j=1}^{N}{\babs{\frac{\sin(\frac{N}{2}(\theta_{i}-\theta_{j}))}{N\sin(\frac{\theta_{i}-\theta_{j}}{2})}}}
$$ 
and 
$$c_{i}:=\sum_{j=1}^{N}{\babs{\frac{\sin(\frac{N}{2}(\theta_{j}-\theta_{i}))}{N\sin(\frac{\theta_{j}-\theta_{i}}{2})}}}.
$$ 
Therefore, it is the case that $r_i = c_i$ and the maximum eigenvalue $\lambda_{N}$ satisfies 
$$
\lambda_{N}\leq\max\{ r_{i}\,:\,i=1,2,\ldots,N\}.
$$ 
For each $i\in\{1,2,\ldots,N\}$ the random variable $r_{i}$ has the same distribution as the random variable
$$
Y=1+X_{1}+X_{2}+\ldots+X_{N-1}
$$ 
where $X_{1},\ldots, X_{N-1}$ are independent identically distributed random variables conditioned to the phase $\theta_{i}$. Moreover, each $X_{k}$ has the same distribution as 
$$
X(\theta):=\babs{\frac{\sin(\frac{N}{2}\theta)}{N\sin(\frac{\theta}{2})}}\hspace{0.1cm}\text{where $\theta$ is distributed accordingly to $\nu$}
$$
conditioned on the phase $\theta_i$. 
\noindent We will assume that $\nu$ is a probability measure on $[-\pi,\pi]$ which is absolutely continuous with respect to Lebesgue measure and with bounded pdf $f\in L^{\infty}([-\pi,\pi])$. In what follows the Chernoff bound construction will only make 
use of $\norm{f}_{\infty}$ and hence is independent of $\theta_i$.

\vspace{0.3cm}
\begin{remark}
The random variable $X(\theta)$ has expectation 
$$\mathbb{E}(X(\theta))=\int_{-\pi}^{\pi}{\babs{\frac{\sin(\frac{N}{2}\theta)}{N\sin(\frac{\theta}{2})}}f(\theta)\,d\theta}$$

\noindent and second moment 
$$\mathbb{E}(X(\theta)^{2})=\int_{-\pi}^{\pi}{\babs{\frac{\sin(\frac{N}{2}\theta)}{N\sin(\frac{\theta}{2})}}^{2}f(\theta)\,d\theta}.$$

\noindent In case $\nu$ is the uniform distribution then $f(\theta)=\frac{1}{2\pi}$ and using the results on the integral of the Dirichlet kernel (see \cite{Dirichlet}) we see that: 
$$\mathbb{E}(X(\theta))\asymp \frac{\log N}{N}.$$

\noindent Using Parseval's Theorem it is straightforward to see that $$\mathbb{E}(X(\theta)^{2})=\frac{1}{N}.$$
\end{remark}

\vspace{0.3cm}

\begin{theorem}\label{upper_teo}
Given $\epsilon>0$ and $\kappa \in (0,1)$, then for every $u\geq 0$, there exists $N=N(\kappa,\epsilon)$ such that for all
$ N > N(\kappa,\epsilon)$
\begin{equation}
\mathbb{P}\Big(\lambda_{N}\geq (C+\epsilon)\log N +u\Big)\leq K\frac{e^{-u}}{N^{\epsilon\kappa}}
\end{equation}

\noindent where $K>0$ is a constant independent on $\epsilon,\kappa, u$ and $N$ and $C=\big(4\pi\norm{f}_{\infty}(e-1)+1\big)$.
\end{theorem}

\vspace{0.2cm}
\noindent Note that for $\epsilon = 1+\delta$ with  $\delta >0$ the result implies that 
$\lambda_N > (C+\epsilon)\log N$ occurs finitely many times a.s. by the Borel-Cantelli Lemma. 

\vspace{0.2cm}
\noindent As a Corollary we have:
\vspace{0.3cm}

\begin{cor}
 \begin{equation}
\mathbb{E}(\lambda_{N})\leq  \Big(4\pi\norm{f}_{\infty}(e-1)+1\Big)\log N + o(1).
\end{equation}
\end{cor}

\noindent We will first prove Theorem \ref{upper_teo}. 

\vspace{0.2cm}
\begin{proof}
It is easy to see that
\begin{equation}\label{eq_char} 
\babs{\frac{\sin(\frac{N}{2}\abs{x})}{N\sin(\frac{\abs{x}}{2})}}\leq\sum_{k=1}^{N/2}{\,\frac{1}{k}\cdot\,\mathbf{1}_{[\frac{2\pi(k-1)}{N},\frac{2\pi k}{N})}(\abs{x})}
\end{equation}

\noindent for every $x\in [-\pi,\pi]$. Let us define $p_{k}=\nu([\frac{2\pi(k-1)}{N},\frac{2\pi k}{N}])$ and $p_{-k}:=\nu([-\frac{2\pi k}{N},-\frac{2\pi(k-1)}{N}])$ and let $q_{k}=p_{k}+p_{-k}$ for every $k=1,2,\ldots, \frac{N}{2}$.

\noindent Hence, for every $t\in [0,1]$
$$\mathbb{E}(e^{Xt}|\theta_i)\leq \sum_{k=1}^{N/2}{q_{k}e^{\frac{t}{k}}}\leq 1+(e-1)\sum_{k=1}^{N/2}{q_{k}\frac{t}{k}}$$

\noindent where the first inequality comes from (\ref{eq_char}) and the last inequality comes from the fact that $e^{t}\leq 1+(e-1)t$ for every $t\in [0,1]$. Since $\max_{k}\{p_{k},p_{-k}\}\leq \frac{2\pi}{N}\norm{f}_{\infty}$ we conclude that 
\begin{equation}
\mathbb{E}(e^{X}|\theta_i)\leq 1+\frac{4\pi(e-1)\norm{f}_{\infty}}{N}\cdot H_{N/2}
\label{eqn_condlmom}
\end{equation}
\noindent where $H_{p}:=\sum_{k=1}^{p}{\frac{1}{k}}$ is the Harmonic series.

\vspace{0.2cm}
\noindent The random variables $r_{i}$ all have the same distribution as
$$
Y=1+X_{1}+\ldots+X_{N-1}
$$ 
where $\{X_{1},\ldots,X_{N-1}\}$ are i.i.d. conditioned on the phase $\theta_i$. 
However, the $r_i$ are not independent. Using the Chernoff bound with the 
random variable $Y$, unconditioning and setting $t=1$ (see \cite{Feller_Vol1}) 
we see that:
\begin{eqnarray*}
\mathbb{P}\Big(Y &\geq & C\log(N)+u | \theta_i\Big) \leq  \mathbb{E}(e^{X}|\theta_i)^{N}\cdot 
e^{-(C_\delta\log N + u)}\\
   &\leq &\frac{e^{-u}}{N^{C_\delta}}\cdot\Bigg( 1+\frac{4\pi(e-1)\norm{f}_{\infty}}{N}\cdot H_{N/2} \Bigg)^{N}.
\end{eqnarray*}
This follows from the bound on the conditional moment generating function (see equation \ref{eqn_condlmom}). 
Here $C_\delta = C - \delta$ and $\delta >0$. Let $N_{\delta}$ be the minimum positive integer such that the following inequality holds $\delta \log N_\delta >1$. Then for all $N \geq N_\delta$ and unconditioning we see that
$$
\mathbb{P}\Big(Y \geq  C\log(N)+u \Big) \leq \frac{e^{-u}}{N^{C_\delta}}\cdot\Bigg( 1+\frac{4\pi(e-1)\norm{f}_{\infty}}{N}\cdot H_{N/2} \Bigg)^{N}.
$$

\noindent For any positive function $f:\mathbb{N}\to [0,+\infty)$ 
\begin{eqnarray*}
\Bigg( 1+\frac{f(N)}{N}\Bigg)^{N} & = & \mathrm{exp}\Bigg[ N\log\Bigg(1+\frac{f(N)}{N}\Bigg)\Bigg] \\
 & \leq & \mathrm{exp}(f(N))
\end{eqnarray*}

\noindent where the last inequality comes from the fact that $\log(1+x)\leq x$ for all $x\geq 0$. 

\vspace{0.3cm}
\noindent Hence, 
\begin{equation*}
\Bigg(1+\frac{4\pi(e-1)\norm{f}_{\infty}}{N}\cdot H_{N/2} \Bigg)^{N} \leq\mathrm{exp}\Bigg[ 4\pi(e-1)\norm{f}_{\infty}H_{N/2}\Bigg]
\end{equation*}

\noindent and since $H_{p}\leq\log{p}+\gamma+\half p^{-1}$ where $\gamma$ is the Euler--Mascheroni constant $\gamma\approx 0.57721566...$ 
\begin{equation*}
 \mathrm{exp}\Bigg[ 4\pi(e-1)\norm{f}_{\infty}H_{N/2}\Bigg]\leq K\cdot \mathrm{exp}\Bigg[ 4\pi\norm{f}_{\infty}(e-1) \log(N/2) \Bigg]
\end{equation*}

\noindent where $K=\mathrm{exp}(4\pi\norm{f}_{\infty}(e-1)\gamma)$.

\noindent Let us define the random variable $Z$ as $Z=\max_{k}\{r_{k}\}$. Therefore, 
\begin{eqnarray*}
\mathbb{P}(\lambda_{N}\geq C\log N+u) & \leq & \mathbb{P}(Z\geq C\log N+u)\\
  & \leq & N\mathbb{P}(Y\geq C\log N+u)
\end{eqnarray*}

\noindent where the last inequality comes from the union bound. Hence, for all $u\geq 0$ we have 
\begin{eqnarray*}
\mathbb{P}(\lambda_{N}\geq C\log N+u) & \leq & \frac{N}{N^{C_\delta}}Ke^{-u}\Bigg( \frac{N}{2}\Bigg)^{4\pi\norm{f}_{\infty}(e-1)} \\
  & \leq & K_{2}e^{-u}N^{(4\pi\norm{f}_{\infty}(e-1)+1-C_\delta)}.
\end{eqnarray*}

\noindent Let $\epsilon>0$ and  $\kappa \in (0,1)$ be given. Set $\delta = \epsilon(1-\kappa) > 0$ and set $C=4\pi\norm{f}_{\infty}(e-1)+1+\delta$ so that for every $u\geq 0$ and $N > N_{\delta}$ where $\delta=\epsilon\kappa$
$$\mathbb{P}(\lambda_{N}\geq (C+\epsilon)\log N +u)\leq K_{2}\frac{e^{-u}}{N^{\epsilon\kappa}}.$$
\end{proof}

\noindent Now we follow with the proof of the Corollary.

\vspace{0.2cm}

\begin{proof} Using the previous Theorem we have that,
\begin{eqnarray*}
\mathbb{E}(\lambda_{N}) & = & \int_{0}^{+\infty}{\mathbb{P}(\lambda_{N}\geq t)\,dt} \leq \int_{0}^{C\log N}{\mathbb{P}(\lambda_{N}\geq t)\,dt} \\
& + & \int_{0}^{+\infty}{\mathbb{P}(\lambda_{N}\geq C\log N+u)\,du} \\
& \leq & \Big(4\pi\norm{f}_{\infty}(e-1)+1\Big)\log N + \frac{K_{2}}{N^{\epsilon/2}},
\end{eqnarray*}
for sufficiently large $N$.

\noindent Since $K_{2}$ does not depend on $N$ and $\epsilon$ is arbitrary we conclude that for $N$ sufficiently large 
\begin{equation}
\mathbb{E}(\lambda_{N})\leq  \Big(4\pi\norm{f}_{\infty}(e-1)+1\Big)\log N + o(1).
\end{equation}

\end{proof}

\begin{remark}
For the case $\nu$ is uniform distribution on $[-\pi,\pi]$ we have that 
\begin{equation}
\mathbb{E}(\lambda_{N})\leq  \Big(2e-1\Big)\log N + o(1).
\end{equation}
\end{remark}

\begin{remark}
We believe that the constant $C$ in Theorem \ref{upper_teo} can be sharpened by working with the optimal choice for $t$.
\end{remark}

\subsection{Lower Bound}
\label{sec_maxlower}
\noindent The main result we will derive in this Section is the following.
\vspace{0.1cm}
\begin{theorem}\label{lower_teo}\label{thm_maxlowbnd}
Let $\nu$ be an absolutely continuous probability distribution on $[-\pi,\pi]$ with continuous probability density. Let $\lambda_N$ be the maximum eigenvalue of the random matrix ${\bf V^{*}V}$ generated accordingly to $\nu$. Then for any $0<\alpha<1$
\begin{equation}
\label{eqn_lowliminf}
\mathbb{P}\Bigg(\lambda_N \geq \frac{%\half
\alpha\cdot\log N}{\log\log N}\Bigg)=1-o(1).
\end{equation}
\end{theorem}

\noindent In the proof we will need the following result which was proved in \cite{Raab}.
\vspace{0.3cm}

\begin{theorem}
\label{lemma_ballsurns}
Let $a>0$ and $b>0$ and suppose that there are $aN$ balls and $bN$ urns, and we throw the balls independently and uniformly at random in the urns. Let $M$ be the random variable that counts the maximum number of balls in any urn. Then $\mathbb{P}(M>\alpha\cdot k_{N})=o(1)$ if $\alpha>1$ and $\mathbb{P}(M>\alpha\cdot k_{N})=1-o(1)$ if $0<\alpha<1$, where $k_N = \log N/\log \log N$.
\end{theorem}

\vspace{0.3cm}
\noindent For additional results along these lines, see \cite{Raab}. We would like to remark that the estimates in \cite{Raab} can be also derived via the maxima of unit Poisson random variables and a natural construction for the occupancy experiment.

\vspace{0.3cm}
\begin{proof}
Let $b$ be a positive integer to be chosen later. Divide the interval $[0,2\pi]$ into $bN$ intervals of the same length $\frac{2\pi}{bN}$. These intervals will represent the urns and the angles $\theta_{1},\theta_{2},\ldots,\theta_{N}$ will represent the balls we will throw into the urns accordingly to the distribution $\nu$.

We will now develop a lower bound for the maximum eigenvalue, $\lambda_{N}$. Let $f(x)$ be the continuous function such that $d\nu(x)=f(x)dx$. Let $x_{0}$ be such that $f(x_{0})>0$. Given $\delta>0$ let $\epsilon>0$ be such that $\abs{f(x_{0})-f(y)}<\delta$ for all $y$ such that $\abs{x_{0}-y}<\epsilon$. Let $K_{\epsilon}$ be defined as 
$$
K_{\epsilon}=\int_{-\epsilon}^{\epsilon}{f(x_{0}+t)\,dt}>0.
$$
Using the Strong Law of large numbers we know that as $N$ increases we
will have $K_{\epsilon}N$ angles in the interval
$(x_{0}-\epsilon,x_{0}+\epsilon)$ with probability $1-o(1)$. Choosing
$\delta$ sufficiently small we can assume without loss of generality
that $f$ is constant in this interval. Since we divide the interval
$[0,2\pi]$ into $bN$ intervals we know that $\frac{\epsilon}{\pi}bN$
of these intervals will lie inside
$(x_{0}-\epsilon,x_{0}+\epsilon)$. Therefore, we have the game were we
throw $K_{\epsilon}N$ balls into $\frac{\epsilon}{\pi}bN$ urns with
uniform distribution. Therefore, using Theorem \ref{lemma_ballsurns}
with high probability we will find at least $K=\alpha\cdot k_{N}$
distinct $\theta_{k_1},\cdots,\theta_{k_K}$ with the property for
\begin{equation}
\label{eqn_cluster}
\vert \theta_{k_j} -  \theta_{k_\ell} \vert < \frac{2\pi}{bN},
~~1 \leq k,\ell \leq K
\end{equation}
where, of course the constraint holds when the two indices are the same.

\vspace{0.3cm}
\noindent Consider now the square real symmetric submatrix with rows and columns 
corresponding to the above indices, $k_1,\cdots,k_K$. Let $S$ be this matrix. Note that up to reenumeration of the angles we can assume without loss of generality that $S$ is the $K\times K$ principal minor. Let $F_{N}$ be the function defined by
$$
F_N(x) = \frac{\sin \Big(\frac{N}{2} x\Big)}{N\sin \Big(\frac{x}{2}\Big )}.
$$
Let $\gamma>0$ and small and choose $b$ sufficiently large so 
$$
F_{N}(\frac{2\pi}{bN})>1-\gamma
$$ 
for all $N$. The entries of $S$ are all positive since
$$
S_{k,\ell} = F_N(\theta_k - \theta_\ell) > F_N\Bigg(\frac{2\pi}{bN}\Bigg) \geq 1-\gamma.
$$
Let $\lambda_{S}$ be the maximum eigenvalue of this matrix. By standard Perron-Frobenius theory this eigenvalue is positive and satisfies,
$$
\lambda_{S} \geq K\cdot F_N\Bigg(\frac{2\pi}{bN}\Bigg) \geq (1-\gamma)K.
$$
A further standard result, see \cite{wilkinson}, is that the eigenvalues of a real symmetric matrix and its square submatrices interlace. It
thus follows that for any $0<\alpha<1$ with high probability we have:
\begin{equation}
\label{eqn_informalbound}
\lambda_{N} \geq \lambda_{S} \geq (1-\gamma)\alpha\cdot\frac{\log N}{\log\log N} 
\end{equation}

\noindent More precisely, 
\begin{equation*}
\label{eqn_lowliminfB}
\mathbb{P}\Bigg(\lambda_N \geq \frac{(1-\gamma)
\alpha\cdot\log N}{\log\log N}\Bigg)=1-o(1).
\end{equation*}
Now since $\gamma$ is arbitrary we proved our Theorem.
\end{proof}

\subsection{Remarks on the Upper Bound}

In this Section we would like to observe that the results obtained in Theorem \ref{upper_teo} are not valid if the pdf is unbounded. Consider the probability density given by the following pdf:
\begin{equation}\label{al_pdf}
p_{\alpha}(x)=\frac{1-\alpha}{2\pi^{1-\alpha}}\cdot\frac{1}{\abs{x}^{\alpha}}\cdot\mathbf{1}_{[-\pi,\pi]}(x). 
\end{equation}
Then if we define $p_{N}$ as
$$
p_{N}:=\int_{-\frac{\pi}{2N}}^{\frac{\pi}{2N}}{p_{\alpha}(x)\,dx}=\Big(\frac{1}{2N}\Big)^{1-\alpha}.
$$
By the Strong Law of Large numbers the expected number of the angles $\theta_{1},\ldots, \theta_{N}$ in the interval $[-\frac{\pi}{2N},\frac{\pi}{2N}]$ is $\frac{N^{\alpha}}{2^{1-\alpha}}$. Therefore, repeating the same argument given in Theorem \ref{lower_teo} we have:

\vspace{0.2cm}
\begin{theorem}
Let $0<\alpha<1$ and let $p_{\alpha}$ be the pdf given in equation (\ref{al_pdf}) then if we consider ${\bf V}$ the random Vandermonde matrix constructed according to $p_{\alpha}$ and let $\lambda_{N}$ be the maximum eigenvalue of ${\bf V^{*}V}$ we have that
$$\mathbb{P}\Bigg(\lambda_N \geq \half\cdot\frac{N^{\alpha}}{2^{1-\alpha}}\Bigg)=1-o(1).$$
\end{theorem}

\section{Conjectured Lower Bound on the $K_{\rho,u}$}
\label{sec_lowerbnd}
Given $\rho\in\mathcal{P}(n)$ we would like to find a lower bound for $K_{\rho,u}$ in terms on $n$ and the size of the blocks of $\rho$. 
This will immediately give us a lower bound on $m_{n}$. But first let us fix some notation and review some preliminary results. Following the 
definition in \cite{Feller_Vol2}, $( x )_+^k$ is 0 
if $ x \leq 0$ and is $x^k$ otherwise. The density of the sum of $m$ i.i.d. uniform distributions in the interval $[-\frac{1}{2},\frac{1}{2}]$ is  
$$ 
g^{(m)}(t) = \frac{1}{(m-1)!}\sum_{k=0}^{m-1} (-1)^k \perm{m}{k}\lb t + \frac{m}{2} - k \rb^{(m-1)}_+  .
$$
The following Lemma will be used subsequently.
\begin{lemma}
\label{lem_unif_depend}
Suppose that $p \leq m \leq n$ then the following inequality holds
$$\int_{-\infty}^\infty g^{(m-p)}(t) g^{(p)}(t)g^{(n-p)}(t)\,dt \geq g^{(m)}(0) g^{(n)}(0)$$
\end{lemma}

\noindent See Appendix for the proof of this result.

%Lemma \ref{lem_unif_depend} has the following generalization in terms
%of random walks. Consider random walks generated by a discrete
%uniform - $N$ random variable. In particular consider a random walk with $p$
%steps $X_1+ \cdots + X_p$ which is continued with two independent segments of
%lengths $m-p, n-p$. Then,
%\begin{eqnarray*}
%\prob{X_1+ \cdots + X_p + Y_1 + \cdots + Y_{m-p} = 0;\\ 
%X_1+ \cdots + X_p + Z_1 + \cdots + Z_{n-p} =0}\\
%\geq \prob{X_1+ \cdots + X_p + Y_1 + \cdots + Y_{m-p} = 0} \times \\
%        \prob{X'_1+ \cdots + X'_p + Z_1 + \cdots + Z_{n-p} = 0} \\
%\end{eqnarray*}

%\noindent here $X'_1,\cdots,X'_p$ are independent of $X,Y,Z$. 
\vspace{0.2cm}
\begin{remark}\label{r2}
The sequence $\{g^{(n)}(0)\}_{n}$ is decreasing and using the Central Limit Theorem we can prove that $g^{(n)}(0)\asymp h_{n}=\frac{\sqrt{6}}{\sqrt{\pi}}\frac{1}{\sqrt{n}}.$ We also want to emphasize that this approximation is good even for small values of $n$,
\begin{eqnarray*}
g^{(2)}(0)=1 & , & h_{2}=\frac{\sqrt{3}}{\sqrt{\pi}}\approx 0.9772\\
g^{(3)}(0)=\frac{3}{4} & , & h_{3}=\frac{\sqrt{2}}{\sqrt{\pi}}\approx 0.7979\\
g^{(4)}(0)=\frac{2}{3} & , & h_{4}=\frac{\sqrt{3}}{\sqrt{2\pi}}\approx 0.6910
\end{eqnarray*}
\end{remark}

\vspace{0.3cm}

As it was noticed in \cite{GC02}, each partition $\rho\in\mathcal{P}(n)$ with $r$ blocks determines a set of $r$ equations $E_{1},E_{2},\ldots,E_{r}$ in $n$ variables $M_{1},M_{2},\ldots,M_{n}$. These equations have rank $r-1$ and satisfy that 
$$
E_{1}+E_{2}+\ldots+E_{r}=0.
$$ 
Repeating the analysis done in Appendix B of \cite{GC02} we know that $K_{\rho,u}$ can be expressed as probability (or volume) of the solution set 

\begin{equation}
 E_{j}\,:\,\sum_{i\in B_{j}}{M_{i}}=\sum_{i\in B_{j}}{M_{i-1}},
\end{equation}

\noindent with $-\frac{1}{2}\leq M_{k}\leq\frac{1}{2}$ independent and uniform random variables. It was also observed in \cite{GC02} that the probability of this set is always a rational number small or equal than 1 and that $K_{\rho,u}=1$ if and only if the partition is non--crossing.

\subsection{Case r=2}
In this case we have only two blocks and
$\rho=\{B_{1},B_{2}\}$. Therefore we have $n$ i.i.d. uniform
distributed random variables $\{M_{k}\}_{1\leq k\leq n}$ and
$K_{\rho,u}$ is the probability of the set that satisfies the
equations:
$$E_{1}\,:\,\sum_{i\in B_{1}}{M_{i}}=\sum_{i\in B_{1}}{M_{i-1}}$$
and
$$E_{2}\,:\,\sum_{i\in B_{2}}{M_{i}}=\sum_{i\in B_{2}}{M_{i-1}}.$$

\vspace{0.3cm}

\noindent Since $E_{1}+E_{2}=0$ we need only to satisfy one of them. Without loss of generality we can assume that $k=|B_{1}|\leq |B_{2}|=n-k$ and let $2m$ be the set of variables appearing in both sides of equation $E_1$ (after possible cancellation of variables). Note that $m\leq k$ and $m=k$ if there is no cancellation in $E_{1}$. Then 
$$
K_{\rho,u}=(G^{m}\ast G^{m})(0)
$$
where $G(x)=\frac{\sin(\pi x)}{\pi x}$. Since the continuous Fourier transform of the normalized sinc function (to ordinary frequency) is $g(t)=\mathbf{1}_{[-\half,\half]}(t)$ we have that 
$$
K_{\rho,u}=\int_{-\infty}^{+\infty}{(g^{(m)})^{2}(t)\,dt}=g^{(2m)}(0)
$$
\noindent where 
$$
g^{(m)}(t)=\underbrace{(g\ast g\ast\ldots\ast g)}_{m}(t).
$$
Since the sequence $\{g^{(n)}(0)\}_{n}$ is strictly decreasing we have 
$$K_{\rho,u}=g^{(2m)}(0)\geq g^{(2|B_{1}|)}(0)$$

\subsection{Case r=3}

Let us first motivate this case with an example. Consider partition $\rho\in\mathcal{P}(10)$ given by $\rho=\{\{1,5,7\},\{3,9,10\},\{2,4,6,8\}\}$. We only need to consider two of the three equations. Let us consider the ones associated with the smallest blocks, in this case $B_{1}$ and $B_{2}$. Then the corresponding equations are:
$$E_{1}\,:\,M_{1}+M_{5}+M_{7}=M_{10}+M_{4}+M_{6}$$
\noindent and 
$$E_{2}\,:\,M_{3}+M_{9}+M_{10}=M_{2}+M_{8}+M_{9}.$$

After cancellation and arranging common variables to both equations on the same side we obtain:
$$M_{10}=M_{1}+M_{5}+M_{7}+N_{4}+N_{6}$$
\noindent and
$$M_{10}=M_{2}+M_{8}+N_{3}$$
\noindent where $N_{j}=-M_{j}$. Note that by the way we arrange the variables there is no common variable in the RHS of the previous equations. Therefore, it is clear that
$$K_{\rho,u}=\int_{-\infty}^{+\infty}{g(t)g^{(5)}(t)g^{(3)}(t)\,dt}.$$

\noindent Therefore, using Lemma \ref{lem_unif_depend} we see that 
$$K_{\rho,u}\geq g^{(6)}(0)g^{(4)}(0).$$

\noindent Note that $g^{(6)}(0)g^{(4)}(0)$ is equal to product of the 
probability densities for the set of i.i.d. uniform random variables
$X_{1},\ldots,X_{6},Y_{1},\ldots,Y_{4}$ satisfying equations
$X_{1}+\ldots+X_{6}=0$ and $Y_{1}+\ldots+Y_{4}=0$. Since the sequence
$\{g^{(n)}(0)\}_{n}$ is strictly decreasing we proved that
$$K_{\rho,u}\geq g^{(2|B_{1}|)}(0)g^{(2|B_{2}|)}(0).$$

\vspace{0.3cm}

It is clear that the same argument described in the previous example can be carried out in general. Hence, if $\rho\in\mathcal{P}(n)$ with 3 blocks such that $|B_{1}|\leq |B_{2}|\leq |B_{3}|$ then 
$$K_{\rho,u}\geq g^{(2|B_{1}|)}(0)g^{(2|B_{2}|)}(0).$$

\subsection{Conjecture on the General Case}

\noindent Several numerical simulations and examples strongly suggest that the following conjecture is true.

\vspace{0.2cm}
\begin{conjecture}\label{tlb}
Let $\rho\in\mathcal{P}(n)$ be a partition of $n$ elements with $r$ blocks of size $n_{1}\leq n_{2}\leq \ldots\leq n_{r}$. Then 
$$K_{\rho,u}\geq \prod_{j=1}^{r-1}{g^{(2n_{j})}(0)}.$$ 
\end{conjecture}

\noindent If the previous conjecture holds then using Remark \ref{r2} we have that for $n$ sufficiently large
\begin{equation}\label{eqlab2}
K_{\rho,u}\geq\prod_{j=1}^{r-1}{g^{(2n_{j})}(0)}\approx\prod_{j=1}^{r-1}{h_{2n_{j}}}=\Bigg(\frac{6}{\pi}\Bigg)^{\frac{r-1}{2}}\prod_{j=1}^{r-1}
{\frac{1}{\sqrt{2n_{j}}}}.
\end{equation}
By the Arithmetic and Geometric mean inequality 
$$
n_{1}\ldots n_{r-1}\leq \Bigg(\frac{n_{1}+\ldots+n_{r-1}}{r-1}\Bigg)^{r-1}\leq \Bigg(\frac{n}{r-1}\Bigg)^{r-1}.
$$ 
Therefore,
$$
\Bigg(\frac{r-1}{n}\Bigg)^{\frac{r-1}{2}}\leq \frac{1}{\sqrt{n_{1}\cdots n_{r-1}}}.
$$
\noindent Hence, plugging the previous inequality in equation (\ref{eqlab2}) we obtain
\begin{equation}
K_{\rho,u}\geq \Bigg(\frac{3(r-1)}{\pi n}\Bigg)^{\frac{r-1}{2}}.
\end{equation}

Let us define $L(k,n):=\Big(\frac{3(k-1)}{\pi n}\Big)^{\frac{k-1}{2}}$. Then the previous conjecture implies that $K_{\rho,u}\geq L(r,n)$ for any partition of $n$ elements with $r$ blocks. On the other hand, it was observed in \cite{SP} and \cite{GC02} that $K_{\rho,u}=1$ if and only if the partition is non--crossing. Let $T(k,n)$ be the Narayana numbers i.e., the numbers of non--crossing partition of $n$ elements in $k$ blocks. It is known that 
$$T(k,n)=\frac{1}{k}{n-1 \choose k-1}{n \choose k-1}$$
see \cite{On_en}. Therefore applying the previous analysis we have the following lower bound for the $n$--th moment:
\begin{equation}\label{lowbound}
C_{n}+\sum_{k=1}^{n}{(s(k,n)-T(k,n))\cdot L(k,n)}\leq m_{n} 
\end{equation}
where 
$$
s(k,n)=\frac{1}{k!}\sum_{i=0}^{k}{(-1)^i{k \choose i}(k-i)^{n}}
$$ 
is the Stirling number of the second kind and $C_{n}=\frac{1}{n+1}{2n \choose n}$ is the Catalan number which counts the number of non--crossing partitions of $n$ elements. They also arise as the moments of the semicircular distribution as in (\ref{semicircular}).

\noindent Assuming that Conjecture \ref{tlb} is true and using Harper's Theorem \cite{Harper} it can then be shown that the following Theorem holds.
\vspace{0.3cm}
\begin{theorem}
 Let $B_{n}$ be the $n$--th Bell number, $\sigma_{n}=\big(B_{n+2}/B_{n}-(B_{n+1}/B_{n})^{2}-1\big)^{\frac{1}{2}}$ then
\begin{equation}\label{eqlb3}
 \frac{1}{\sqrt{2\pi}}\cdot\frac{B_{n}}{\sigma_{n}}\cdot\Bigg( \frac{3}{\pi e \log{n}} \Bigg)^{\frac{n}{2e\log{n}}}\leq m_{n}\leq B_{n}.
\end{equation}

\noindent Moreover,
\begin{equation}
 \lim_{n}{\Bigg(\frac{m_{n}}{B_{n}}\Bigg)^{\frac{1}{n}}}=1.
\end{equation}
\end{theorem}

\begin{proof}
Since $m_{n}=\sum_{\rho\in\mathcal{P}(n)}{K_{\rho,u}}$ and since $L(B,n)\leq K_{\rho,u}\leq 1$ for any partition with $B$ blocks we see that $m_{n}\leq B_{n}$ (as was already noted in \cite{GC02}) and 
$$\sum_{j=1}^{n}{s(j,n)L(j,n)}\leq m_{n}$$
which is a weaker inequality than Equation (\ref{lowbound}).
\noindent It is clear that $s(J_{n},n)L(J_{n},n)\leq \sum_{j=1}^{n}{s(j,n)L(j,n)}$. Using Harper's Theorem we see that 
$$s(J_{n},n)\asymp \frac{1}{\sqrt{2\pi}}\cdot\frac{B_{n}}{\sigma_{n}},\hspace{0.3cm}J_{n}\asymp \frac{n}{e\log{n}}$$

\noindent and 
$$L(J_{n},n)\asymp\Bigg( \frac{3(J_{n}-1)}{\pi n}\Bigg)^{\frac{J_{n}-1}{2}}\asymp \Bigg( \frac{3}{\pi e \log{n}}\Bigg)^{\frac{n}{2e\log{n}}}$$ 

\noindent Hence proving equation (\ref{eqlb3}).

\vspace{0.3cm}
To prove the second part of the Theorem we first note that since $m_{n}\leq B_{n}$ then 
$$\liminf_{n}{\Big(\frac{m_{n}}{B_{n}}\Big)^{\frac{1}{n}}}\leq 1.$$ 

\noindent On the other hand, using equation (\ref{eqlb3}) we have
$$m_{n}^{\frac{1}{n}}\geq \Bigg(\frac{1}{\sqrt{2\pi}}\Bigg)^{\frac{1}{n}}\cdot \Bigg(\frac{3}{\pi}\Bigg)^{\frac{1}{2e\log n}}\cdot \Bigg(\frac{B_{n}}{\sigma_{n}}\Bigg)^{\frac{1}{n}}\cdot \Bigg(\frac{1}{e\log{n}}\Bigg)^{\frac{1}{2e\log{n}}}.$$

\noindent Since $\lim_{n}\Big(\frac{1}{\sqrt{2\pi}}\Big)^{\frac{1}{n}}=\lim_{n}\Big(\frac{3}{\pi}\Big)^{\frac{1}{2e\log n}}=1$ and 
$$\lim_{n}{ \Big(\frac{1}{e\log{n}}\Big)^{\frac{1}{2e\log{n}}}}   =\lim_{n}{\mathrm{exp}\Big( -\frac{\log\log n}{\log n}\Big)}=1.$$

\noindent We have that 
$$\limsup_{n}{\Bigg(\frac{m_{n}}{B_{n}}\Bigg)^{\frac{1}{n}}}\geq \limsup_{n}{\Bigg(\frac{1}{\sigma_{n}}\Bigg )^{\frac{1}{n}}}=1$$
\noindent where the last equality comes from Remark \ref{r3}. 
\end{proof}

%\begin{conjecture}
%Several numerical computations and examples strongly suggest that the lower bound $L(B,n)$ for the expansion coefficient $K_{\rho, u}$ can be improved to 
%\begin{equation}
%\widehat{L}(B,n)=\underbrace{G^{\lfloor \frac{n}{B} \rfloor}\ast\ldots\ast G^{\lfloor \frac{n}{B} \rfloor}}_{B}(0)=\int_{\mathbb{R}}
%{\big( g^{(\lfloor \frac{n}{B} \rfloor)}\big )^{B}(t)\,dt}.
%\end{equation}
%This number corresponds to the Vandermonde expansion coefficient of the totally interlaced partition. See Figure \ref{fig_n6B2}. However we are yet to settle %this question. Figure \ref{fig_n6B2} shows the conjectured bound achieving partition
%for the case $n=6$ and $B=2$. Figure \ref{fig_n9B3} shows the corresponding
%\begin{figure}[!Ht]
% \begin{center}
%    \includegraphics[width=0.5\columnwidth]{part_n6_B2}
%    \caption{Conjectured Partition for $n=6$ and $B=2$}
%   \label{fig_n6B2}
%  \end{center}
%\end{figure}
%partition for $n=9$ and $B=3$.
%\begin{figure}[!Ht]
%  \begin{center}
%   \includegraphics[width=0.5\columnwidth]{part_n9_B3}
%    \caption{Conjectured Partition for $n=9$ and $B=3$}
%    \label{fig_n9B3}
% \end{center}
%\end{figure}
%\end{conjecture}

\section{Capacity of the Vandermonde Channel}
\label{applications}
Consider the Gaussian matrix channel \cite{emre} in which the received signal $ y \in \cmplx^N$ is given as 
\begin{equation}
y = {\bf H} x + z
\end{equation}
where $ z\sim {\cal CN}({\bf 0},{\bf I}_N)$,  $x \sim {\cal
CN}({\bf 0},{\bf I}_L)$, and ${\bf H}\in
\cmplx^{N \times L}$ has i.i.d. zero mean Gaussian entries and is
{\em standard} see \cite{Verdu}.  
%It is known that if ${\bf H}$ is a $N\times N$ standard, complex,
%Gaussian matrix, 
Then an explicit expression for the asymptotic capacity exists \cite{Verdu}:
$$
\lim_{N\to\infty}{\frac{1}{N}\log\mathrm{det}\Big(I_{N}+\gamma{\bf HH^{*}}\Big)}=-\frac{\log e}{4\gamma}F(\gamma,\beta)
$$
$$
+\,\,\beta\log\Big( 1+\gamma -\frac{1}{4}F(\gamma,\beta) \Big) + \log\Big(1+\beta\gamma -\frac{1}{4}F(\gamma,\beta) \Big) \\
$$
where
$$
F(a,b):=\Big( \sqrt{a\lb 1 + \sqrt{b}\rb^2+1}-
\sqrt{a\lb 1-\sqrt{b}\rb^2+1}\Big)^2
$$
and the SNR $\gamma$ is 
$$
\gamma = \frac{N\cdot\Exp{ \vert \vert {\bf x} \vert\vert^2}}{L\cdot
\Exp{ \vert \vert {\bf z} \vert\vert^2}}
$$
and the ratio $\frac{N}{L}\to\beta$ as $N \rightarrow \infty$.

\vspace{0.2cm}
\noindent We can prove that a similar limit exists and is finite if
the Gaussian matrix ${\bf H}$ is replaced with a random
Vandermonde matrix. Moreover, using Jensen's inequality we can get an
upper bound on the capacity. More precisely, if we fix $\gamma > 0$ to
be the SNR, we may  define $C_{V}(\gamma)$ the asymptotic capacity of the
Vandermonde channel (whenever the limiting moments exist and 
define a measure) for random Vandermonde matrices 
${\bf V} \in \cmplx^{N \times L}$  to be 
\begin{eqnarray}
\label{eqn_caplim}
C_{V}(\gamma) & := &
\lim_{N\to\infty}{\mathbb{E}\Bigg(\frac{1}{N}\log\mathrm{det}\Big(I_{N}+\gamma
{\bf VV^{*}}\Big)}\Bigg) \\
& = &\lim_{N\to\infty}{\mathbb{E}\Bigg(\frac{1}{N}\log\mathrm{det}\Big(I_{L}
+\gamma {\bf V^{*}{\bf V}}\Big)}\Bigg) \nonumber\\
& = & \lim_{N\to\infty}{\mathbb{E} \Big(\mathrm{tr}_N \log \Big( {\bf I}_L + \gamma {\bf
V^{*}{\bf V}}\Big)} \Big)  \nonumber \\
& = & \lim_{L\to\infty}\int_{0}^\infty{c\log(1+\gamma t)\,d\mu_{L}(t)}\nonumber\\
& = & \int_{0}^\infty{c\log(1+\gamma t ) d\mu_c(t) } \nonumber
%& = & \lim_{N\to\infty}{\int_{0}^{N}{\log_{2}(1+\gamma t)\,d\mu_{N}(t)}}\\
%& = & \int_{0}^{\infty}{\log_{2}(1+\rho t)\,d\mu(t)}
\end{eqnarray}
where $\mu_{L}$ is the empirical measure of the $L\times L$ random
matrix ${\bf V^{*}V}$ and $\mu_c$ is the limit measure of the $\mu_{L}$. The
first equality follows from Sylvester's Theorem on determinants, the
second and third are by definition, and the final equality is a
consequence of their uniform integrability.  This latter  follows
from  $\log \lb 1 + \gamma t \rb < \gamma t, t > 0$   and that given
$\varepsilon > 0$,
$\exists \alpha > 0$ such that
$$
\sup_L \int_\alpha^\infty t d\mu_L(t) < \varepsilon
$$
see the converse statement in \cite{Billingsley68} Theorem 5.4.

\vspace{0.2cm}
\noindent Therefore, by Jensen's inequality
\begin{eqnarray*}
C_{V}(\gamma) & = & c\int_{0}^{\infty}{\log(1+\gamma t)\,d\mu_{c}(t)}\\
& \leq &  c\log(1+\gamma).
\end{eqnarray*}
since the limit first moment is 1. 

\vspace{0.2cm}
\noindent As an application of the above consider the network with $M$ mobile users conducting synchronous
multi-access to a base station with $N$ antenna elements, arranged
as a uniform linear array, \cite{Krim}. The network is depicted in
Figure \ref{fig_ulanetwork}, showing the directions of arrival of
two mobiles A and B.
\begin{figure}[!Ht]
  \begin{center}
\hspace{-0.5in}
    \includegraphics[width=3.2in]
{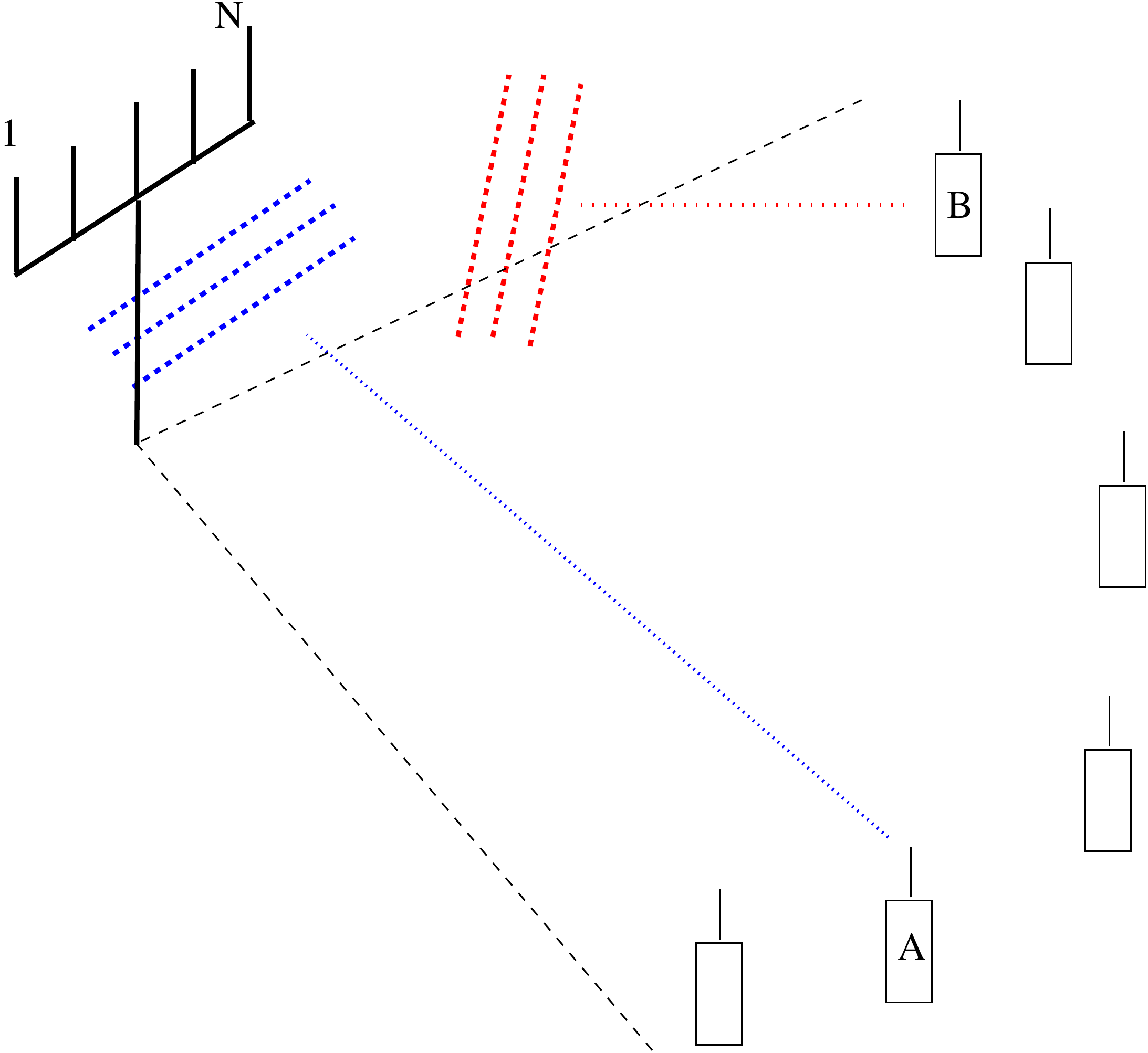}
    \caption{Multi access network with scheduling with $M$ mobile users conducting synchronous
multi-access to a base station with $N$ antenna elements, arranged
as a uniform linear array.}
    \label{fig_ulanetwork}
  \end{center}
\end{figure}

\vspace{0.2cm}
\noindent Suppose a random subset of $L$  users in any time slot are selected to transmit. Then the antenna array response over 
the selected users is ${\bf V}$ equals to
\begin{equation}
\label{eqn_ula}
\frac{1}{\sqrt{N}}
\lb
\begin{array}{ccc}
1  & \cdots & 1 \\
e^{-2 \pi i \frac{d}{\lambda} \sin(\theta_1)} & \cdots & e^{-2 \pi i
\frac{d}{\lambda} \sin(\theta_L)} \\
\vdots & \vdots & \vdots \\
e^{-2 \pi i (N-1) \frac{d}{\lambda} \sin(\theta_1)} & \cdots & e^{-2 \pi
i(N-1) \frac{d}{\lambda} \sin(\theta_L)}
\end{array}
\rb
\end{equation}
where $d$ is the element spacing and $\lambda$ is the wavelength, see
also \cite{GC02}. Let us suppose that $M,L,N$ are large
and that the   
angles of arrival are uniformly scattered in $\lb -\alpha, \alpha
\rb$. Then it is reasonable to determine performance supposing that
the angles of arrival are drawn uniformly so that the maximum sum
throughput (equivalently per user rate) is determined by
(\ref{eqn_caplim}) with the phase pdf given as,
$$
q_{\alpha}(\theta)=\frac{1}{2\alpha\sqrt{\frac{4\pi^2d^2}{\lambda^2}-\theta^2}} 
$$ 
for $\theta\in [-\frac{2\pi d\sin\alpha}{\lambda},\frac{2\pi
d\sin\alpha}{\lambda}]$.  The above example is largely illustrative
and a more realistic one would depict a network of users with
heterogeneous received powers corresponding to varying mobile
distances. Such generalizations as this and others can be treated
using the results given here but we do not go into details.

In Figure \ref{fig_capcurve} we show the capacity for uniform phase and with the phase distribution density
$q_\alpha(\theta)$ corresponding to the uniform linear array with 
$\lambda=2d$ and various values of $\alpha$.

\section{Numerical Results}
\label{sec_numerfin}
\subsection{Simulations for the Limiting Measure}
In this Section we present some numerical results and simulations. In Figure \ref{fig_unifpdf} we show the simulation results obtained
for 700 independent samples of a square matrix with $N=1000$ rows.
The plot shows the sample histogram for the eigenvalues.
The results strongly suggest that there is an atom at 0 in
the limiting measure. As can be seen there is very little probability
mass above 4 which is consistent with our bounds for the maximum
eigenvalue.
\begin{figure}[!Ht]
  \begin{center}
%\hspace{-0.5in}
    \includegraphics[width=3.5in, height=3.0in]%[width=\figwidth]
{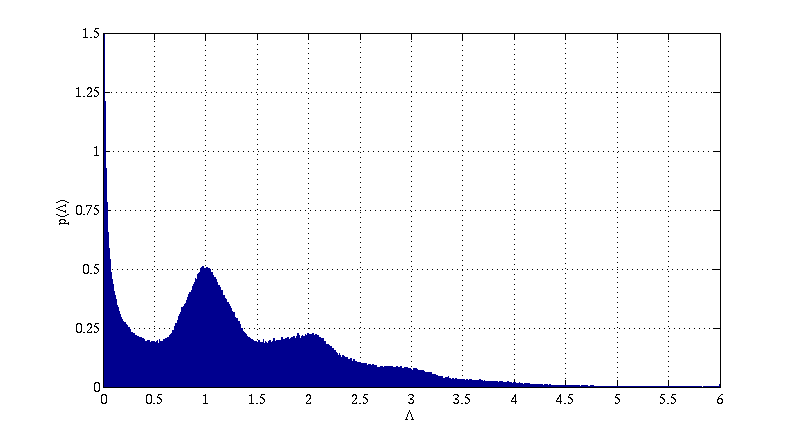}
    \caption{Simulated limit distribution for $\theta \sim U[-\pi,\pi]$ with $L=N=1000$ averaged over 700 sample matrices.}
    \label{fig_unifpdf}
  \end{center}
\end{figure}

In the following Figure we consider the pdf given in equation
(\ref{eqn_goodpdf}) and the simulation described above was 
repeated. Our results are shown in Figure \ref{fig_logeigpdf}.
\begin{figure}[!Ht]
  \begin{center}
%\hspace{-0.5in}
    \includegraphics[width=3.5in, height=3.0in]
{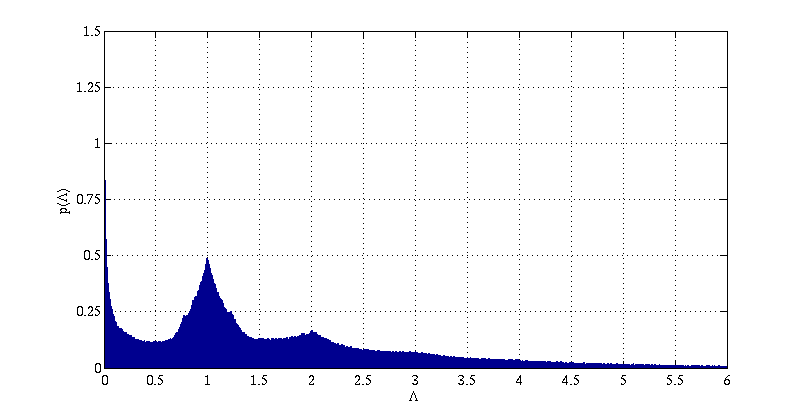}
    \caption{Simulated limit distribution for $\theta \sim f(x)$, see (\ref{eqn_goodpdf}), with $L=N=1000$ averaged over 700 sample matrices.}
    \label{fig_logeigpdf}
  \end{center}
\end{figure}
The results are similar to the uniform case, in that an atom at 0
is strongly suggested. Also the pdf possesses peakes at 1 and 2
as was found in the uniform case. It should be noted that the right
tail is more spread out than before. 

\subsection{Moment Bounds}

Figure \ref{fig_mombnd} depicts the upper and lower bound for the limiting moments $m_n$ as derived in Section 
\ref{sec_lowerbnd}. The horizontal axis is the moment index and the vertical axis is the (natural) log of
the bound.  
\begin{figure}[!Ht]
  \begin{center}
%\hspace{-0.5in}
    \includegraphics[width=3.5in]%[width=\figwidth]
{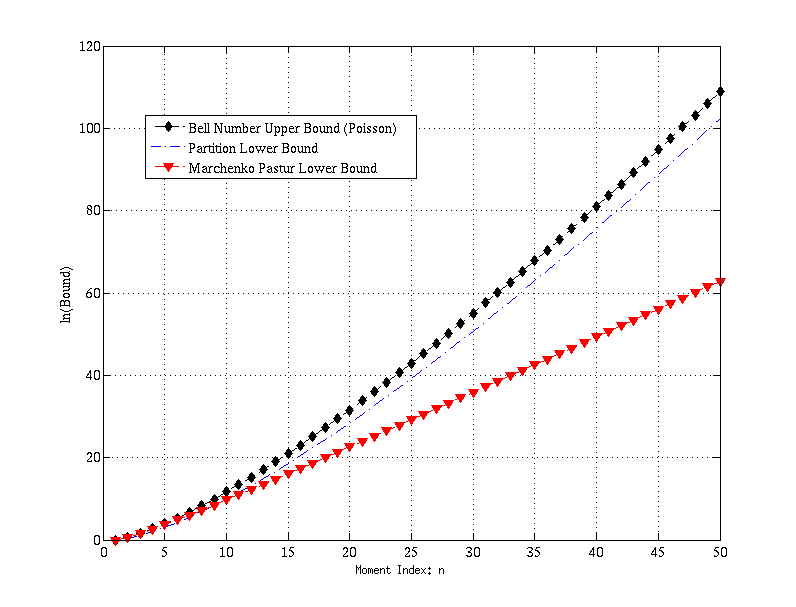}
    \caption{Upper and lower bounds for $m_{n}$.}
    \label{fig_mombnd}
  \end{center}
\end{figure}
As the Figure shows, the eigenvalue moments lie much more closely
to the Bell Number upper bound than they do to the bound
corresponding to the Marchenko-Pastur distribution (MP), see also
\cite{GC02}. Close inspection shows that the partition lower
bound lies belows  MP. This is because the asymptotic
formula, based on the central limit theorem, is being used as opposed
to the bound constructed from the peaks of uniform
densities. The moments themselves are not depicted.

In the following we use the more accurate version of the bound see Table \ref{table_mombnd}.
\begin{table}[Hht]
 \begin{center}
  \caption{Moments and Moment Bounds, $d=1$ (Vandermonde)}
  \begin{tabular}{|c|c|c|c|c|} \hline
    $n$  & $C_n$ & $L_n$ & $m_n$    & $B_n$ \\ \hline
     4   & 14    &  $14\frac{2}{3}$ &  $14\frac{2}{3}$ & 15    \\ \hline
     5   & 42    &  $48\frac{2}{3}$ &  $48\frac{2}{3}$ & 52    \\ \hline
     6   & 132   &  $\approx 176.2944$ & 178.55   & 203\\ \hline
     7   & 429    & $\approx 684.4611$      &   713.66667  & 877    \\ \hline
  \end{tabular}
  \label{table_mombnd}
 \end{center}
\end{table} 
Here $n$ is the moment index, $C_n$ is the $n$--th Catalan number, $L_n$
is our lower bound, $m_n$ the moment and $B_n$ the Bell number upper
bound. 

We now present the corresponding results for the empirical eigenvalue measure
from \cite{SP} when $d=2$. As before,  we use the more accurate version
of the bound see table \ref{table_sigmombnd} and with the same
notation. The exact moments are obtained by squaring the values of
the crossing partitions. For large $d$ the Marchenko-Pastur limit is 
approached as the contribution of the crossing partitions goes to 0. 
\begin{table}[Hht]
 \begin{center}
  \caption{Moments and Moment Bounds}
  \begin{tabular}{|c|c|c|c|c|} \hline
    $n$  & $C_n$ & $L_n$ & $m_n$    & $B_n$ \\ \hline
     4   & 14    &  $14\frac{4}{9}$ &  $14\frac{4}{9}$ & 15    \\ \hline
     5   & 42    &  $46\frac{4}{9}$ &  $46\frac{4}{9}$ & 52    \\ \hline
     6   & 132   &  $\approx 160.0928$ & 162.6358   & 203\\ \hline
     7   & 429    & $\approx 579.1567$  & 619.6256  & 877    \\ \hline
  \end{tabular}
  \label{table_sigmombnd}
 \end{center}
\end{table} 
As can be seen in this example the lower bound, $L_n$, once again provides a much more
accurate estimate of the moments than $C_n$.

\subsection{Maximum Eigenvalue}
We now present results for the maximum eigenvalue $\lambda_N$. Figure
\ref{fig_maxeigunif} shows the theoretical upper and lower
bounds, together with simulated results for $\theta \sim U(-\pi,\pi)$.
\begin{figure}[!Ht]
  \begin{center}
%\hspace{-0.5in}
    \includegraphics[width=3.5in]%[width=\figwidth]
{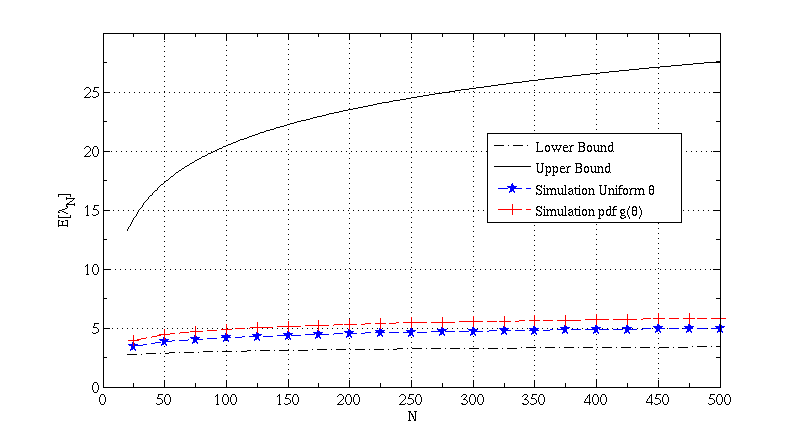}
    \caption{Upper and lower Bounds for the maximum eigenvalue with
    simulation results.}
    \label{fig_maxeigunif}
  \end{center}
\end{figure}
The results are for matrices with up to $N=500$ and the sample mean
values of the maximum eigenvalue were obtained using 10,000 random
matrices per point. As can be seen the results follow the lower bound
closely. In a second experiment the bounded pdf (with two
discontinuity points) was simulated, 
$$
g(\theta):= \frac{2}{\lb \pi \sqrt{\pi^2-\theta^2}\rb}, \,\,\,\,\,\theta \in
[-\frac{\pi}{\sqrt{2}},\frac{\pi}{\sqrt{2}}].
$$
The unbounded pdf $p_\alpha(x)$ see equation (\ref{eqn_badpdf}) was also
simulated for $\alpha = \frac{1}{3}, \frac{1}{2} , \frac{2}{3}$ and $0.9$. The results are shown in
Figure \ref{fig_unbndmaxeig}.
\begin{figure}[!Ht]
  \begin{center}
%\hspace{-0.5in}
    \includegraphics[width=3.5in]%[width=\figwidth]
{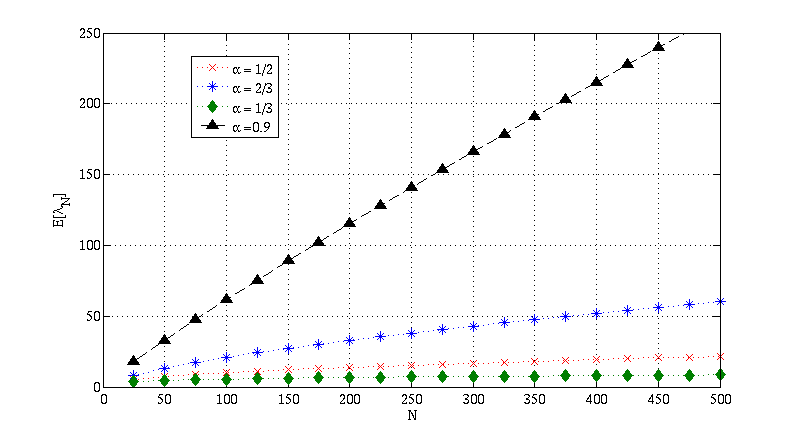}
    \caption{Maximum eigenvalue for the distribution $p_\alpha(\theta)$.}
    \label{fig_unbndmaxeig}
  \end{center}
\end{figure}

\subsection{Results for the Capacity}
We now consider the capacity of the limiting Vandermonde channel, shown to exist in Section \ref{applications}. Simulations were conducted for the uniform distribution and for the distribution with density
$$
q_\alpha(\theta)=\frac{1}{2\alpha\sqrt{\frac{4\pi^2d^2}{\lambda^{2}}-\theta^{2}}}
$$ 
with $\alpha = \frac{\pi}{8},\frac{\pi}{4}$ and $\frac{\pi}{3}$ and with square matrices size $N=100$ and $\lambda=2d$. The distribution $q_\alpha$ with $\alpha = \frac{\pi}{4}$ is depicted in Figure \ref{fig_q_alppdf}. The results are graphed in Figure \ref{fig_capcurve}
\begin{figure}[!Ht]
  \begin{center}
%\hspace{-0.5in}
    \includegraphics[width=3.5in]%[width=\figwidth]
{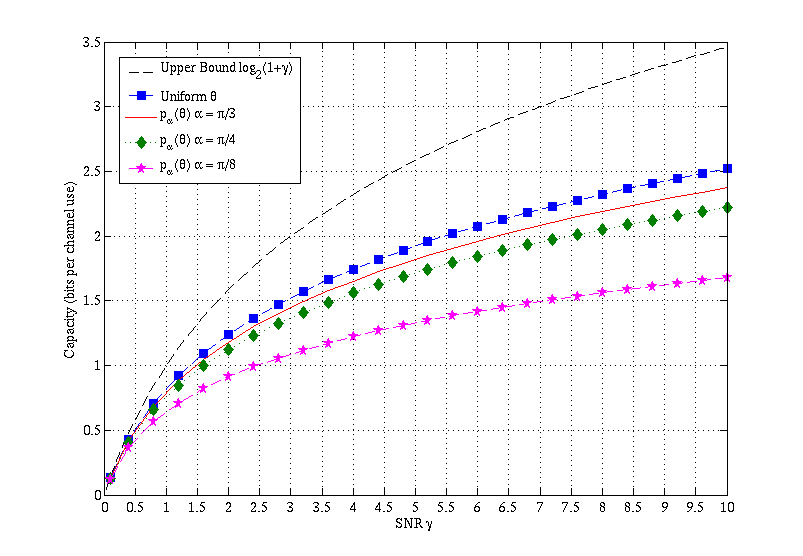}
    \caption{Capacity per channel use per degree of freedom.}
    \label{fig_capcurve}
  \end{center}
\end{figure}
and show that the capacity diminishes as $\alpha$ becomes smaller. For the largest value $\alpha =\frac{\pi}{3}$ capacity is very close to that
determined by the uniform which we included for reference. All curves lie below the Jensen inequality bound, $\log_2(1+\gamma)$.
\begin{figure}[!Ht]
  \begin{center}
%\hspace{-0.5in}
    \includegraphics[width=3.5in]%[width=\figwidth]
{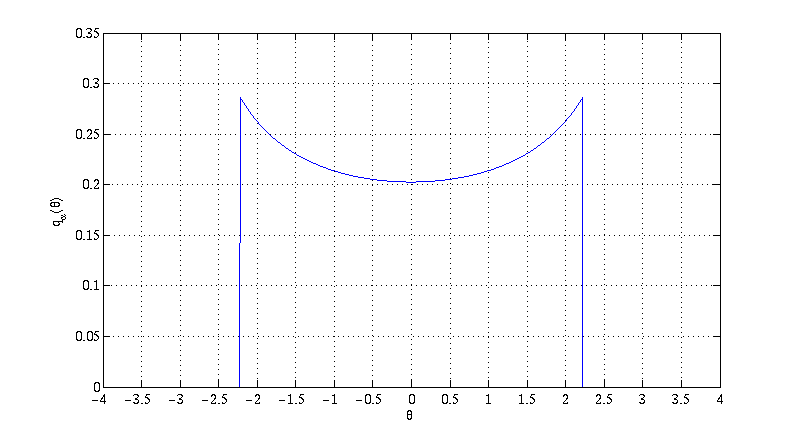}
    \caption{Distribution $q_{\alpha}(\theta)$ with $\lambda=2d,\alpha=\frac{\pi}{4}$.}
    \label{fig_q_alppdf}
  \end{center}
\end{figure}

\section{Conclusions}%
\label{sec_conc}%
The main results of this paper are upper and lower bounds for the maximum eigenvalue of random unit magnitude Vandermonde matrices. Additionally, the limit measure was shown to exist for a wide class of densities. When the densities were not continuous it was required that the Fourier series coefficients be non--negative. This latter condition while it is sufficient appears far from necessary. 

Additional results include the existence of the capacity of the Vandermonde channel (limit integral for the expected log determinant). We also provide evidence to support a conjecture on the lower bound on the Vandermonde expansion coefficients.  We also show that this conjecture implies a tight lower bound for the moment sequence.

As far as the limit eigenvalue distribution is concerned we conjecture that there is an atom at $0$ and not other atoms. The size of this atom will of course depend on the distribution of the phases. This matter can be investigated numerically by examination of the fraction of eigenvalues near 0 as the matrix dimensions go to infinity. 

\appendices
\section{Elementary Lemma and Corollary regarding the eigenvalues of ${\bf V^{*}V}$}
We show that there is a real symmetric matrix which has the
same eigenvalues as ${\bf V^{*}V}$ where ${\bf V}$ is the random Vandermonde matrix.
\vspace{0.3cm}
\begin{lemma}
\label{lemma_equiv}
Let ${\bf A}$ and ${\bf B}$ be $N\times N$ complex matrices and suppose that
\begin{equation}
\label{eqn_component}
B_{k,\ell} = \frac{c_k}{c_\ell} A_{k,\ell}
\end{equation}
with $c_k \neq 0$ for every $k=1,2,\ldots, N$. Then ${\bf A}$ and ${\bf B}$ have the same eigenvalues.
\end{lemma}

\begin{proof}
Suppose that $\lambda$ is an eigenvalue of ${\bf A}$ with corresponding eigenvector ${\bf v}$. Set the vector ${\bf u}$ as
$$
u_\ell = c_\ell v_\ell.
$$
Multiplying ${\bf u}$ by ${\bf B}$ gives,
$$
\sum_\ell B_{k,\ell} u_\ell = c_k \sum_\ell A_{k,\ell} v_\ell = \lambda c_k v_k 
= \lambda u_k.
$$
Thus any eigenvalue of ${\bf A}$ is also an eigenvalue of ${\bf B}$. Since also
\begin{equation}
\label{eqn_AtoB}
A_{k,\ell} = \frac{d_k}{d_\ell} B_{k,\ell}
\end{equation}
with $d_k = 1/c_k \neq 0$ and $k=1,\ldots,N$, we find that all the eigenvalues of ${\bf B}$ are also eigenvalues of ${\bf A}$ mutatis mutandis.
\end{proof}

\vspace{0.3cm}
\begin{cor} \label{corr_complexreal}
The matrix ${\bf X_{N}}$ with entries
$$
X_{i,j} = \Bigg( \frac{\sin(\frac{N}{2}(\theta_{i}-\theta_{j}))}{N\sin(\frac{\theta_{i}-\theta_{j}}{2})}\Bigg)
$$
has the same eigenvalues as ${\bf V^{*}V}$.
\end{cor}

\begin{proof}
Let $a_{kl}$ be the $(k,l)$ entry of the matrix ${\bf V^{*}V}$ then it is easy to see that
\begin{equation}
 a_{kl}=\frac{1}{N}\sum_{p=0}^{N-1}{e^{ip(\theta_{k}-\theta_{l})}}.
\end{equation}
Hence 
\begin{eqnarray*}
\label{eqn_geometric}
1 + \cdots + e^{-i(N-1)x} & = & \frac{e^{-iNx/2}}{e^{-ix/2}}
\cdot \lb \frac{e^{iNx/2} - e^{-iNx/2}}{e^{ix/2} - e^{-ix/2}} \rb \\
 & = & \frac{e^{-iNx/2}}{e^{-ix/2}} \cdot \frac{\sin(Nx/2)}{\sin(x/2)}.\nonumber
\end{eqnarray*}
Substituting,
$$ 
x = \theta_k - \theta_\ell 
$$
the first term on the RHS of the previous equation is 
$$ 
\frac{e^{iN\theta_\ell/2} e^{i\theta_k/2}}{e^{iN\theta_k/2}e^{i\theta_\ell/2}} 
$$
\noindent and has the form given in Lemma \ref{lemma_equiv}.
\end{proof}

\section{The proof of Lemma \ref{lem_unif_depend}}
\begin{proof}
By definition,
$$
\int_{-\infty}^{+\infty}{g^{(m-p)}(t) g^{(p)}(t)\,dt} = g^{(m)}(0)
$$
\noindent and similarly for $n$. Also the functions $g^{(m_1)}(t),g^{(m_2)}(t)$
increase and decrease together for any integers $m_1,m_2 \geq 0$. It follows
immediately from the FKG inequalities \cite{Alon} that,
$$\int f(t) h(t) \mu(t)\, dt \geq \int f(t) \mu(t) dt \cdot \int h(t) \mu(t)\,dt$$
where $\mu$ is a finite measure and $f,h$ increase and decrease together.
That the result holds, following on substituting 
\begin{eqnarray*}
f(t) & = & g^{(m-p)}(t) \\
h(t) & = & g^{(n-p)}(t) \\
\mu(t) & = & g^{(p)}(t) 
\end{eqnarray*}
\end{proof}

\section{The proof of Theorem \ref{t1}}
\label{sec_prooft1}
\begin{proof}
Let $\rho\in\mathcal{P}(n)$ be a partition of $n$ with $r$ blocks. Let $\omega$ be a probability measure with density $p(x)$. Let
$E_{i}=\sum_{k\in B_{i}}{i_{k-1}-i_{k}}$ where $B_{i}$ is the $i$-th block. We are required to show that,
\begin{equation}
 K_{\rho,\omega}=\lim_{N}{K_{\rho,\omega,N}}
\end{equation}
where 
\begin{eqnarray}
\label{eqn_traceeqn}
K_{\rho,\omega,N} & = & \sum_{{\bf m}}{G^\rho_{N}({\bf m})\prod_{k=1}^{r}{\mathbb{E}(e^{im_{k}\theta_{k}})}}
\end{eqnarray}
where 
$$
{\bf m} :=  (m_{1},\ldots,m_{r})\in\mathbb{Z}_{0}^{r}
$$
with
$$ 
\mathbb{Z}_{0}^{r} :=  \{(z_{1},\ldots,z_{r})\in\mathbb{Z}^{r}\,:\,z_{1}+\ldots+z_{r}=0\}
$$
and
$$
G^\rho_{N}({\bf m})  :=  \frac{N^{r}}{N^{n+1}}\cdot S^{\rho,{\bf m}}_N
$$
and 
$$ 
S^{\rho,{\bf m}}_N := \#\{{\bf i}=(i_1,\ldots,i_n)\,:\,E_{1}=m_{1},\ldots,E_{r}=m_{r}\}.
$$
Note that Equation (\ref{eqn_traceeqn}) is the expected trace evaluated for the partition $\rho$. 

\vspace{0.2cm}
\noindent In the special case where the phases $\theta$'s are uniform distributed on $[-\pi ,\pi]$, the coefficients of $G_N^\rho({\bf m})$ are all 0 except when ${\bf m}$ = 0. It follows then that,
$$
K_{\rho,u} = \lim_{N\rightarrow \infty} G_N^\rho({\bf 0})
$$
is the expansion coefficient of the uniform.  
\vspace{0.2cm}
\noindent To obtain the limit we proceed via the Lebesgue dominated convergence theorem, \cite{Wil91}, and note we
are summing over ${\bf m} \in  \mathbb{Z}_{0}^{r}$. 

\vspace{0.2cm}
\noindent As far as pointwise convergence is concerned, by Lemma \ref{lemma_t1} and for 
any $(m_{1},\ldots,m_{r})\in\mathbb{Z}_{0}^{r}$ the limit exists
$$
\lim_{N}G^\rho_{N}({\bf m})=K_{\rho,u} \in (0,1]
$$
independent of ${\bf m} = (m_{1},\ldots,m_{r})$.

\vspace{0.2cm}
\noindent As far as the dominated part is concerned, first note that
$$
0 \leq G_N^\rho({\bf m}) \leq 1
$$
for all {\bf m} so we only need to show that 
\begin{equation}
%\label{eqn_dominate}
\sum_{{\bf m}} \prod_{k=1}^B \vert a(m_k) \vert < \infty,
\end{equation}
where 
$$
a(m):=\mathbb{E}(e^{im\theta_{k}})=\int_{-\pi}^{\pi}{e^{imx}p(x)\,dx},
$$
since the phases $\theta_{1},\ldots,\theta_{N}$ are independent and identically  distributed. By hypothesis 
$a(m)\geq 0$,  thus convergence alone is sufficient. This is demonstrated below. An example of an unbounded pdfs on $[-\pi,\pi]$ which meet the Theorem's conditions is given in Lemma \ref{lemma_logpdf}. 

\vspace{0.2cm}
\noindent Let $\widehat{p}(n)=\frac{1}{2\pi}a(n)$. Using the fact that $p\in\mathcal{L}$ and applying Theorem 15.22 of \cite{cham} repetedly we see that 
\begin{eqnarray*}
\sum_{(m_{1},\ldots,m_{B})\in\mathbb{Z}_{0}^{r}}{\,\,\prod_{k=1}^{r}{a(m_{k})}} & = & \underbrace{a\ast a\ast\ldots\ast a}_{r}\,(0)\\
& = & (2\pi)^{r}\widehat{p^{r}}\,(0)\\
& = & (2\pi)^{r-1}\int_{-\pi}^{\pi}{p^{r}(x)\,dx}.
\end{eqnarray*}
It follows that 
\begin{eqnarray*}
K_{\rho,\omega}=K_{\rho,u}\cdot\sum_{(m_{1},\ldots,m_{r})\in\mathbb{Z}_{0}^{r}}{\,\,\prod_{k=1}^{r}{a(m_{k})}}.
\end{eqnarray*}
Therefore,
$$
K_{\rho,\omega}=K_{\rho,u}\cdot(2\pi)^{r-1}\int_{-\pi}^{\pi}{p(x)^{r}\,dx}.
$$
This finishes the proof.
\end{proof}

\vspace{0.2cm}

\begin{remark}
We observe that (\ref{eqn_dominate}) holds if $\hat{p} \in \ell^1$. This condition however forces the density to be continuous \cite{Edwards}.
\end{remark}

\vspace{0.4cm}
\noindent We now show pointwise convergence for $G_N^\rho({\bf m})$ with ${\bf m}$ fixed. 

\begin{lemma}
\label{lem_partsolve}
Suppose we are given a partition $\rho$ with $n$ variables and $r\geq 2$ blocks
and corresponding partition equations $E_\rho$. Then amongst the
$n$ variables there is a subset 
$$
{\cal F}^\rho_{r-1} \subset \lc 1, \cdots, n \rc
$$
such that $\vert {\cal F}^\rho_{r-1} \vert = r-1$ and for each $f \in {\cal F}^\rho_{r-1}$
$$
M_f = \sum_{k \not{\in} {\cal F}^\rho_{r-1}} F^\rho_{f,k} M_k
$$
where $F^\rho_{f,k} \in \lc -1,0,1 \rc$ and there is one more term with
$1$ than with $-1$. Since the $M_f$ are distinct these equations are
independent. 
\end{lemma}

\begin{lemma}\label{lemma_t1} 
Let $G_{N}(m_{1},\ldots,m_{r})$ be defined as before and $(m_{1},\ldots,m_{r})\in\mathbb{Z}_{0}^{r}$. Then $$\lim_{N\to\infty}G_{N}(m_1,\ldots,m_{r})=\lim_{N\to\infty}G_{N}(0,\ldots,0)=K_{\rho,u}.$$
\end{lemma}
\begin{proof}
Given a $N\times N$ Vandermonde matrix and a partition $\rho$ of $\lc 1,\cdots,n \rc$ with $r$ blocks there is a corresponding
set of partition equations $E_1=0$, $E_2=0,\ldots, E_r=0$ as before. Lemma \ref{lem_partsolve} shows that there is a set of $r-1$ distinct indexes, ${\cal F}^\rho_{r-1}$, such that for all $v \in {\cal F}^\rho_{r-1}$:
\begin{equation}
\label{eqn_F}
M_v = \sum_{k\not{\in} {\cal F}^\rho_{r-1}} F^\rho_{v,k}M_k.
\end{equation}
Let ${\bf m}$ be a vector in $\in \ZBZ$. Now we want to solve a similar problem but where the equations are $E_1=m_1,\ldots,E_r=m_r$ and ${\bf m}=(m_1,\ldots,m_r)$. Indeed, inspection of the proof of Lemma \ref{lem_partsolve} shows us that  
there exist coefficients $\mu_v$ (which are linear combinations of the $m_k$) such that:
\begin{equation}
M_v +\mu_v = \sum_{k\not{\in} {\cal F}^\rho_{r-1}} F^\rho_{v,k}M_k.
\end{equation}
It is convenient to treat the indexes $M_k$ for $k=1,\ldots,n$ as if they were obtained as random i.i.d. uniform in $\lc 1,\ldots,N\rc$, taking each value with probability $1/N$. Define the event
\begin{equation}
\label{eqn_Edef}
{\cal E}_{{\bf m}} :=\{{\bf i}=(i_1,\ldots,i_n)\,:\,E_{1}=m_{1},\ldots,E_{r}=m_{r}\}
\end{equation}
for ${\bf m} \in \ZBZ$. We therefore obtain the following display,
\begin{eqnarray*}
\prob{{\cal E}_{{\bf m}}} & = & \frac{S^{\rho,{\bf m}}_N}{N^n} \\
   & = & \frac{N \cdot G_N^{\rho}({\bf m})}{N^{r}} \\
                & = & \frac{1}{N^{r-1}}\cdot\mathbb{P}\big(M_v+\mu_v \in I_N + \mu_v \,:\, v\in{\cal F}^\rho_{r-1}\big)
\end{eqnarray*}
where $I_N=\lc 1,\ldots,N\rc$ and $\mu_v$ as before. 

\vspace{0.2cm}
\noindent To complete the Lemma it remains to show that the limit,
\begin{equation}\label{eqlemma}
\lim_{N} \mathbb{P} \Big(M_v+\mu_v \in I_N + \mu_v \,:\, v\in{\cal F}^\rho_{r-1}\Big)
\end{equation}
exists and has the stated value. Define the centered and scaled random variables,
$$
U_{N,k}  :=  \frac{\lb M_k -(N-1)/2\rb}{N}
$$
for every $k$. Since the prelimit vector ${\bf U}_N$ has independent components 
with distribution function, 
$$
F_{N,U}(x) = \frac{\floor{Nx+(N-1)/2}+1}{N} \longrightarrow
x + \frac{1}{2}
$$
as $N$ goes to infintity for every $x\in [-\half,\half]$ we may deduce, \cite{Billingsley68} that, 
$$
{\bf U}_N \Rightarrow  {\bf U} \sim U\Bigg(\Big[-\frac{1}{2},\frac{1}{2}\Big]\Bigg)^n
$$
as $N \rightarrow \infty$. 
\vspace{0.2cm}
\noindent It follows that the limit in Equation (\ref{eqlemma}) is equal to
\begin{equation}
\mathbb{P} \Big(U_{N,v} \in\Big[-\half,\half\Big] + \frac{\mu_v}{N} \,:\, v\in{\cal F}^\rho_{r-1}\Big).
\end{equation}
The latter probability converges to 
\begin{equation}
\mathbb{P} \Big( U_{v} \in \Big[-\frac{1}{2},\half\Big]  \,:\, v\in{\cal F}^\rho_{r-1}\Big)
\end{equation}
as $N \rightarrow \infty$, since the $\mu_v$ are fixed. This limit is independent of ${\bf m}$ and thus determines the expansion
coefficient for the uniform as mentioned earlier. The Lemma is proved.  
\end{proof}

\end{document}